\documentclass[10pt,oneside]{amsart}

\usepackage[
paper=a4paper,
headsep=15pt,text={135mm,216mm},centering
]{geometry}

\usepackage[bookmarks]{hyperref}
\usepackage{amssymb,amsxtra}
\usepackage{graphicx,xcolor}
\usepackage{float,array,calc,booktabs,stackrel}
\usepackage{enumitem}\setlist[enumerate,1]{font=\upshape}

\usepackage{tikz}
\usetikzlibrary{
  cd,
  calc,
  positioning,
  arrows,
  decorations.pathreplacing,
  decorations.markings,
}
\usepackage[mode=buildnew]{standalone}

\definecolor{todo-background-color}{gray}{0.95}
\usepackage[
  textwidth=1.1in,
  backgroundcolor=todo-background-color,
  bordercolor=black,
  linecolor=black,
  textsize=footnotesize
]{todonotes}

\iftrue
    
    \makeatletter
    \def\@settitle{%
      \vspace*{-10pt}
      \begin{flushleft}%
        \LARGE\bfseries
        \strut\@title\strut
      \end{flushleft}%
    }
    \def\@setauthors{%
      \begingroup
      \def\thanks{\protect\thanks@warning}%
      \trivlist
      \raggedright
      \large \@topsep27\p@\relax
      \advance\@topsep by -\baselineskip
    \item\relax
      \author@andify\authors
      \def\\{\protect\linebreak}%
      \authors
      \ifx\@empty\contribs
      \else
      ,\penalty-3 \space \@setcontribs
      \@closetoccontribs
      \fi
      \normalfont
      \endtrivlist
      \endgroup
    }
    \def\@setaddresses{\par
      \nobreak \begingroup
      \small\raggedright
      \def\author##1{\nobreak\addvspace\smallskipamount}%
      \def\\{\unskip, \ignorespaces}%
      \interlinepenalty\@M
      \def\address##1##2{\begingroup
        \par\addvspace\bigskipamount\noindent
        \@ifnotempty{##1}{(\ignorespaces##1\unskip) }%
        {\ignorespaces##2}\par\endgroup}%
      \def\curraddr##1##2{\begingroup
        \@ifnotempty{##2}{\nobreak\noindent\curraddrname
          \@ifnotempty{##1}{, \ignorespaces##1\unskip}\/:\space
          ##2\par}\endgroup}%
      \def\email##1##2{\begingroup
        \@ifnotempty{##2}{\nobreak\noindent E-mail address%
          \@ifnotempty{##1}{, \ignorespaces##1\unskip}\/:\space
          \ttfamily##2\par}\endgroup}%
      \def\urladdr##1##2{\begingroup
        \def~{\char`\~}%
        \@ifnotempty{##2}{\nobreak\noindent\urladdrname
          \@ifnotempty{##1}{, \ignorespaces##1\unskip}\/:\space
          \ttfamily##2\par}\endgroup}%
      \addresses
      \endgroup
      \global\let\addresses=\@empty
    }
    \def\@setabstracta{%
      \ifvoid\abstractbox
      \else
      \skip@17pt \advance\skip@-\lastskip
      \advance\skip@-\baselineskip \vskip\skip@
      \box\abstractbox
      \prevdepth\z@ 
      \vskip-28pt
      \fi
    }
    \renewenvironment{abstract}{%
      \ifx\maketitle\relax
      \ClassWarning{\@classname}{Abstract should precede
        \protect\maketitle\space in AMS document classes; reported}%
      \fi
      \global\setbox\abstractbox=\vtop \bgroup
      \normalfont\small
      \list{}{\labelwidth\z@
        \leftmargin0pc \rightmargin\leftmargin
        \listparindent\normalparindent \itemindent\z@
        \parsep\z@ \@plus\p@
        
      }%
    \item[\hskip\labelsep\bfseries\abstractname.]%
    }{%
      \endlist\egroup
      \ifx\@setabstract\relax \@setabstracta \fi
    }

    \def\ps@headings{\ps@empty
      \def\@evenhead{%
        \setTrue{runhead}%
        \normalfont\scriptsize
        \rlap{\thepage}\hfill
        \def\thanks{\protect\thanks@warning}%
        \leftmark{}{}}%
      \def\@oddhead{%
        \setTrue{runhead}%
        \normalfont\scriptsize
        \def\thanks{\protect\thanks@warning}%
        \rightmark{}{}\hfill \llap{\thepage}}%
      \let\@mkboth\markboth
    }\ps@headings

    \def\section{\@startsection{section}{1}%
      \z@{-1.4\linespacing\@plus-.5\linespacing}{.8\linespacing}%
      {\normalfont\bfseries\Large}}
    \def\subsection{\@startsection{subsection}{2}%
      \z@{-.8\linespacing\@plus-.3\linespacing}{.5\linespacing\@plus.2\linespacing}%
      {\normalfont\bfseries\large}}
    \def\subsubsection{\@startsection{subsubsection}{3}%
      \z@{.7\linespacing\@plus.2\linespacing}{-1.5ex}%
      {\normalfont\itshape}}
    \def\paragraph{\@startsection{paragraph}{4}%
      \z@{.7\linespacing\@plus.2\linespacing}{-1.5ex}%
      {\normalfont\itshape}}
    \def\@secnumfont{\bfseries}

    \renewcommand\contentsnamefont{\bfseries}
    \def\@starttoc#1#2{\begingroup
      \setTrue{#1}%
      \par\removelastskip\vskip\z@skip
      \@startsection{}\@M\z@{\linespacing\@plus\linespacing}%
      {.5\linespacing}{
        \contentsnamefont}{#2}%
      \ifx\contentsname#2%
      \else \addcontentsline{toc}{section}{#2}\fi
      \makeatletter
      \@input{\jobname.#1}%
      \if@filesw
      \@xp\newwrite\csname tf@#1\endcsname
      \immediate\@xp\openout\csname tf@#1\endcsname \jobname.#1\relax
      \fi
      \global\@nobreakfalse \endgroup
      \addvspace{32\p@\@plus14\p@}%
      \let\tableofcontents\relax
    }
    \def\contentsname{Contents}
    \def\l@section{\@tocline{2}{.5ex}{0mm}{5pc}{}}
    \def\l@subsection{\@tocline{2}{0pt}{2em}{5pc}{}}
    \makeatother

\fi

\def\to{\mathchoice{\longrightarrow}{\rightarrow}{\rightarrow}{\rightarrow}}
\makeatletter
\newcommand{\shortxra}[2][]{\ext@arrow 0359\rightarrowfill@{#1}{#2}}
\def\longrightarrowfill@{\arrowfill@\relbar\relbar\longrightarrow}
\newcommand{\longxra}[2][]{\ext@arrow 0359\longrightarrowfill@{#1}{#2}}
\renewcommand{\xrightarrow}[2][]{\mathchoice{\longxra[#1]{#2}}%
  {\shortxra[#1]{#2}}{\shortxra[#1]{#2}}{\shortxra[#1]{#2}}}
\makeatother

\makeatletter
\def\addtagsub#1{\let\oldtf=\tagform@\def\tagform@##1{\oldtf{##1}\hbox{$_{#1}$}}}
\makeatother

\makeatletter
\def\Nopagebreak{\@nobreaktrue\nopagebreak}
\makeatother

\makeatletter
\newtheoremstyle{theorem-giventitle}
        {}{}              
        {\itshape}                      
        {}                              
        {\bfseries}                     
        {.}                             
        {\thm@headsep}                             
        {\thmnote{\bfseries#3}}
\newtheoremstyle{theorem-givenlabel}
        {}{}              
        {\itshape}                      
        {}                              
        {\bfseries}                     
        {.}                             
        {\thm@headsep}                             
        {\thmname{#1}~\thmnumber{#3}\setcurrentlabel{#3}}
\newtheoremstyle{definition-giventitle}
        {}{}              
        {}                      
        {}                              
        {\bfseries}                     
        {.}                             
        {\thm@headsep}                             
        {\thmnote{\bfseries#3}}
\def\setcurrentlabel#1{\gdef\@currentlabel{#1}}
\makeatother

\newtheorem{theorem}{Theorem}[section]
\newtheorem{theoremalpha}{Theorem}

\newtheorem{lemma}[theorem]{Lemma}

\theoremstyle{definition}
\newtheorem{definition}[theorem]{Definition}

\newtheorem{remark}[theorem]{Remark}

\newtheorem*{case2'}{Case 2$'$}

\theoremstyle{theorem-giventitle}
\newtheorem{theorem-named}{}
\theoremstyle{theorem-givenlabel}
\newtheorem{theorem-labeled}{Theorem}

\theoremstyle{definition-giventitle}
\newtheorem{definition-named}{}
\newtheorem{conjecture-named}{}
\newtheorem{case-named}{}

\numberwithin{equation}{section}

\def\d{\partial}
\def\Z{\mathbb{Z}}

\def\Q{\mathbb{Q}}
\def\C{\mathbb{C}}

\def\cP{\mathcal{P}}

\def\cT{\mathcal{T}}

\def\tilde{\widetilde}
\def\hat{\widehat}
\def\sm{\smallsetminus}
\DeclareMathOperator\Ker{Ker}

\DeclareMathOperator\Wh{Wh}

\DeclareMathOperator\sign{sign}
\def\spinc{spin$^c$}
\DeclareMathOperator\Spin{Spin}
\def\Spinc{\Spin^c}

\def\lk{\operatorname{lk}}
\def\otimesover#1{\mathbin{\mathop{\otimes}_{#1}}}
\def\cupover#1{\mathbin{\mathop{\cup}_{#1}}}
\def\csum{\mathbin{\#}}

\def\sbmatrix#1{\big[\begin{smallmatrix}#1\end{smallmatrix}\big]}

\def\N{\mathcal{N}}
\def\Bl{\mathop{B\ell}}
\def\rhot{\rho^{(2)}}
\def\osigmat{\bar\sigma^{(2)}}
\makeatletter

\begin{document}

\title{The bipolar filtration of topologically slice knots}

\author{Jae Choon Cha}
\address{
  Center for Research in Topology\\
  POSTECH\\
  Pohang Gyeongbuk 37673\\
  Republic of Korea\quad
  \linebreak
  School of Mathematics\\
  Korea Institute for Advanced Study \\
  Seoul 02455\\
  Republic of Korea
}
\email{jccha@postech.ac.kr}

\author{Min Hoon Kim}
\address{
  Department of Mathematics\\
  Chonnam National University \\
  Gwangju 61186\\
  Republic of Korea
}
\email{minhoonkim@jnu.ac.kr}

\def\subjclassname{\textup{2010} Mathematics Subject Classification}
\expandafter\let\csname subjclassname@1991\endcsname=\subjclassname
\expandafter\let\csname subjclassname@2000\endcsname=\subjclassname

\begin{abstract}
  The bipolar filtration of Cochran, Harvey and Horn presents a
  framework of the study of deeper structures in the smooth
  concordance group of topologically slice knots.  We show that the
  graded quotient of the bipolar filtration of topologically slice
  knots has infinite rank at each stage greater than one.  To detect
  nontrivial elements in the quotient, the proof simultaneously uses
  higher order amenable Cheeger-Gromov $L^2$ $\rho$-invariants and
  infinitely many Heegaard Floer correction term $d$-invariants.
\end{abstract}

\maketitle

\section{Introduction}

Understanding the difference of the topological and smooth categories is among the main objectives of the topological study of dimension~4.
Knot concordance, which may be viewed as the local case of the general disk embedding problem in dimension 4, has been studied extensively from this viewpoint.
Indeed it is known that several questions on 4-manifolds and smoothings can be investigated via concordance.
In the study of smooth concordance, an ultimate goal is to understand the structure of the smooth concordance group of topologically slice knots, which we denote by~$\cT$.
The group $\cT$ measures the gap between the smooth and topological categories.

In the literature, there are remarkable advances in the study of $\cT$, which are achieved by modern 4-manifold technologies.
The first examples of topologically slice knots which are not smoothly slice,
due to Akbulut and Casson, are established as a consequence of the results of Freedman~\cite{Freedman:1982-1,Freedman:1984-1} and Donaldson~\cite{Donaldson:1983-1}.
As an abelian group, $\cT$ is infinitely generated~\cite{Endo:1995-1}, and the 2-torsion subgroup of $\cT$ is infinitely generated~\cite{Hedden-Kim-Livingston:2016-1}.
We do not attempt to provide a complete list of known results, but as further investigation of the structure of $\cT$, we note that results on summands of $\cT$ \cite{Livingston:2004-1, Livingston:2008-1, Manolescu-Owens:2007-1, Hom:2015-1}, and the study of the quotient of $\cT$ modulo the subgroup of Alexander polynomial one knots \cite{Hedden-Livingston-Ruberman:2012-1, Hedden-Kim-Livingston:2016-1} are especially significant.

Nonetheless, our understanding is still far from obtaining a \emph{classification} of~$\cT$.
In fact we believe that known smooth invariants and obstructions, including those from various versions of gauge theory, Heegaard Floer homology and Khovanov homology, are short of classifying~$\cT$; for instance~\cite[Proposition~1.2]{Cochran-Harvey-Horn:2012-1} suggests that a majority of invariants would not be able to see structures in deep part of $\cT$ unless they were coupled better with the fundamental group.
We will discuss this in more details below.

Regarding the above, we remark that the role of the fundamental group is relatively better understood for topological concordance, to provide a framework toward a classification, especially in developments initiated by work of Cochran, Orr and Teichner~\cite{Cochran-Orr-Teichner:1999-1}.
It may be viewed as an obstruction theoretic approach via a filtration and exact sequences involving iterated quotients.
$L^2$-signatures associated with the $n$th derived subgroup of the fundamental group give obstructions at the $n$th stage.
We also remark that obstruction theoretic approaches give beautiful classifications for link homotopy~\cite{Habegger-Lin:1990-1} and Whitney tower concordance~\cite{Conant-Schneiderman-Teichner:2011-1} where Milnor invariants extracted from the fundamental group are key obstructions.

The main result of this paper shows that modern smooth techniques can be combined with fundamental group information more strongly, to detect nontrivial elements in $\cT$ for which the majority of known smooth obstructions vanish.
Our result also provides, from an obstruction theoretic viewpoint for $\cT$, information on what the iterated quotients may look like.

To discuss our result explicitly, we consider the \emph{bipolar filtration}
\[
  \cT \supset \cT_0 \supset \cT_1 \supset \cdots \supset \cT_n \supset
  \cdots \supset \{0\}
\]
introduced by Cochran, Harvey and Horn~\cite{Cochran-Harvey-Horn:2012-1}. Briefly, the filtration reflects definiteness of the intersection form motivated from Donaldson's work, together with fundamental group information related to derived subgroups and the tower techniques of Casson and Freedman for 4-manifolds and work of Cochran, Orr, and Teichner on knot concordance.
The definition of the subgroup $\cT_n\subset \cT$, which is associated with the $n$th derived subgroup, is recalled in Section~\ref{subsection:definition-bipolar-filtration}.

It is noteworthy that the $\tau$-invariant~\cite{Ozsvath-Szabo:2003-1} and the $\epsilon$-invariant~\cite{Hom:2014-1} are trivial on $\cT_0$, and the slice obstructions from the Heegaard Floer correction term invariants~\cite{Manolescu-Owens:2007-1, Owens-Strle:2006-1, Jabuka-Naik:2007-1, Greene-Jabuka:2011-1} vanish on~$\cT_1$~\cite{Cochran-Harvey-Horn:2012-1}.
Also, from results in~\cite{Cochran-Harvey-Horn:2012-1, Ozsvath-Stipsicz-Szabo:2017-1, Ni-Wu:2015-1, Hom-Wu:2016-1}, it follows that the $\nu^+$-invariant~\cite{Hom-Wu:2016-1} and the $\Upsilon$-invariant~\cite{Ozsvath-Stipsicz-Szabo:2017-1} vanish on~$\cT_0$.
This is a rigorous description of the aforementioned claim that a majority of known smooth invariants does not see deep part of~$\cT$.
To understand structures in $\cT_n$ ($n\ge 1$), it seems necessary to combine existing smooth techniques with more information associated with the fundamental group.

Toward a classification of~$\cT$, two questions on the bipolar filtration are fundamental:
\begin{enumerate}[label=(\roman*)]
  \item What is the quotient $\cT_n/\cT_{n+1}$ for each~$n$? Especially, is it nontrivial?
  \item What is the transfinite term $\cT_\omega = \bigcap_{n\ge 0} \cT_n$? Especially, is it nontrivial?
\end{enumerate}

The main focus of this paper is on the first question.
The case of $n=0$, $1$ were the only cases for which the nontriviality of $\cT_n/\cT_{n+1}$ was previously known~\cite{Cochran-Harvey-Horn:2012-1, Cochran-Horn:2012-1}.
We remark that the nontriviality of $\cT_1/\cT_2$ is already striking, since the above smooth invariants do not detect it.
In \cite{Cochran-Harvey-Horn:2012-1} it was achieved by combining Heegaard Floer homology with Casson-Gordon invariants associated with metabelian quotients of the fundamental group.

The main result of this paper establishes that $\cT_n/\cT_{n+1}$ is
large for $n\ge 2$.

\begin{theoremalpha}
  \label{theorem:main-nontriviality-result}
  For each $n\ge 2$, the quotient $\cT_n/\cT_{n+1}$ has infinite rank.
\end{theoremalpha}

In what follows we discuss some aspects of Theorem~\ref{theorem:main-nontriviality-result} and the techniques of the proof.

\subsubsection*{Combining smooth invariants with fundamental group information}

To detect structures unknown to be detected by existing smooth obstructions, our approach simultaneously uses Heegaard Floer correction term $d$-invariants and amenable Cheeger-Gromov $L^2$ $\rho$-invariants associated with higher derived series quotients of the fundamental group.
We remark that a combination of the Cheeger-Gromov invariant over a certain torsion-free solvable group and the $d$-invariant of the 3-fold cyclic branched cover was used earlier in~\cite{Cochran-Harvey-Horn:2012-1}, in order to obtain a weaker result under an additional not-yet-proven hypothesis, which is implied by the homotopy ribbon slice conjecture.
Our improved method can be carried out without the homotopy ribbon type hypothesis.

A key technique we use is to consider, even for a single knot, an infinite family of $d$-invariants associated to branched covers of various degrees, together with the Cheeger-Gromov $\rho$-invariants.
We show that these infinitely many $d$-invariants are all nonzero for our examples, and derive the desired result using this.
Another key ingredient we employ is the amenable signature theorem in~\cite{Cha-Orr:2009-1,Cha:2010-1} for the Cheeger-Gromov invariants over locally $p$-indicable amenable groups.

Given that our approach successfully provides new information on the deep part of the smooth concordance group $\cT$, it seems natural to ask how other various modern smooth techniques can be improved to be coupled more tightly with topological information associated with the fundamental group.
We believe that further investigations along this direction will be intriguing.

\subsubsection*{Satellite constructions of knots}

We explicitly construct knots which are independent (and so generate a free abelian subgroup of infinite rank) in $\cT_n/\cT_{n+1}$ by using \emph{iterated satellite constructions}.
The knots are of the form $R(J,D)$ shown in Figure~\ref{figure:example-knot}, where $D$ means the positive Whitehead double of the right handed trefoil, and $J$ is a pattern knot.

\begin{figure}[H]
  \includestandalone{example-knot}
  \caption{A satellite knot $R(J,D)$.}
  \label{figure:example-knot}
\end{figure}

The pattern knots $J$ are given as follows.
For brevity, denote by $P(K)$ a satellite knot with pattern $P$ and companion~$K$.
Using a fixed pattern $P$ which is defined in Figure~\ref{figure:stevedore-pattern} in Section~\ref{subsection:construction-of-examples} and using an infinite family $\{J_0^i\}_{i=1}^\infty$ of knots described in Section~\ref{section:signature-realization-by-negative-knots},  
the knots $J$ are defined to be $P^{n-1}(J_0^i)=P(P(\cdots P(J_0^i)\cdots))$, obtained by the $(n-1)$st iterated satellite construction.

Satellite constructions are used as a standard tool in many papers on concordance and related subjects.
For instance, see~\cite{Gilmer:1983-1,Livingston:1983-1,Gilmer-Livingston:1992-1,Cochran-Orr-Teichner:2002-1, Cochran-Teichner:2003-1,Friedl-Teichner:2005-1,Cochran-Harvey-Leidy:2009-1,Friedl-Powell:2011-1,Cochran-Davis-Ray:2014-1,Cochran-Harvey-Powell:2017-1,Hedden:2007-1,Hedden-Livingston-Ruberman:2012-1,Hedden-Kirk:2012-1,Levine:2012-1,Hom:2015-1,Hedden-Kim-Livingston:2016-1}.
Especially, interesting new information has been discovered via satellite operators that produce knots generating subgroups of concordance groups which are not straightforward to detect but actually large.
For topological concordance, work of Cochran, Harvey and Leidy on the fractal nature~\cite{Cochran-Harvey-Leidy:2009-2} is notable.
Among remarkable results for the smooth case, Hedden and Kirk showed that Whitehead doubling has infinite rank image in~$\cT$~\cite{Hedden-Kirk:2012-1}, and proposed a conjecture that Whitehead doubling preserves independence in the smooth knot concordance group.

From this viewpoint, we note that the proof of Theorem~\ref{theorem:main-nontriviality-result} shows that
satellite constructions produce a large subgroup even in the deep part of~$\cT$, where the notion of depth is rigorously given by the bipolar filtration.
More precisely, the satellite operator $Q_n=R(P^{n-1}(-),D)$ injects the infinite set of knots $\{J_0^i\}_{i=0}^\infty$ to a linearly independent subset in $\cT_n/\cT_{n+1}$ for $n\ge 2$ (and consequently in~$\cT_n \subset \cT$).
It appears to be an interesting direction to enrich currently available various smooth invariants, to detect large images even in the deep part of $\cT$ for sophisticated satellite operators such as $Q_n$ above.
Our method supports that it would not be completely impossible.

Also, note that the statement of Theorem~\ref{theorem:main-nontriviality-result} tells that the quotients $\cT_n/\cT_{n+1}$ ($n\ge 2$) have a similarity in that they are of the same rank.
The above discussion exhibits more about the similarity from a geometric viewpoint:
for all $n\ge 2$, our $\Z^\infty$ subgroup in $\cT_n/\cT_{n+1}$ is generated by the image of the \emph{same} knots $J_0^i$ under the satellite operator~$Q_n$.
This shows that the fractal nature of the topological knot concordance group proposed and investigated in \cite{Cochran-Harvey-Leidy:2009-2} is also found, at the least partially, in the smooth concordance group~$\cT$ of topologically slice knots.

\subsubsection*{Strategy of the proof and organization of the paper}

The construction of $Q_n(J_0^i)=R(P^{n-1}(J_0^i),D)$ is also closely related to the strategy of the proof of Theorem~\ref{theorem:main-nontriviality-result}\@.
Recall that if one attempted to prove a knot is not slice, it would be natural to start by investigating metabolizers of the Blanchfield pairing.
For the knot $R(P^{n-1}(J_0^i),D)$, it turns out to have exactly two metabolizers: submodules $\langle\alpha_D\rangle$ and $\langle\alpha_J\rangle$ generated by linking circles $\alpha_D$ and $\alpha_J$ of the handles along which $D$ and $J=P^{n-1}(J_0^i)$ are tied in. (See Figure~\ref{figure:example-knot} above, and also Figure~\ref{figure:base-seed-knot} in Section~\ref{section:construction-of-examples-first-step-proof}.)
Indeed, we begin similarly to prove Theorem~\ref{theorem:main-nontriviality-result}, by considering the two possible metabolizers.
In Section~\ref{section:construction-of-examples-first-step-proof}, we recall the definition of the bipolar filtration, describe the construction of the above satellite knots in detail, and divide the proof into the two cases.
In Sections~\ref{section:using-cheeger-gromov} and~\ref{section:signature-realization-by-negative-knots}, we present the proof for the case of $\langle\alpha_D\rangle$, using amenable Cheeger-Gromov $\rho$-invariants.
In Sections~\ref{section:using-d-invariant} and~\ref{section:the-cobordism-W}, we treat the case of $\langle\alpha_J\rangle$, using infinitely many Heegaard Floer $d$-invariants.

\subsubsection*{Comparison with the link case}

We remark that for the multi-component link case, the nontriviality of $\cT_n/\cT_{n+1}$ was proven earlier by the first named author and Powell~\cite{Cha-Powell:2014-1}.
They built a geometric operation which systematically pushes certain links nontrivial in $\cT_{n-1}/\cT_{n}$ to links nontrivial in the next stage~$\cT_n/\cT_{n+1}$, using covering link calculus.
This works only for links, since the covering link technique requires multi-components.
The approach used in this paper for knots is of a completely different nature.
The main results (Theorems~1.1 and~1.2) of \cite{Cha-Powell:2014-1} for $n\ge 2$ can be obtained as immediate consequences of our Theorem~\ref{theorem:main-nontriviality-result}.


\subsubsection*{Acknowledgements}

We would like to thank anonymous referees for comments which were very helpful in improving the exposition of this paper.
Part of this work was done during the authors' visit to the Max Planck Institute for Mathematics in Bonn.
JCC and MHK were partly supported by NRF grant 2019R1A3B2067839.
MHK was partly supported by POSCO TJ Park Science Fellowship and NRF grant 2021R1C1C1012939. 

\section{Examples and the first step of the proof}
\label{section:construction-of-examples-first-step-proof}

\subsection{Definition of the bipolar filtration}
\label{subsection:definition-bipolar-filtration}

We begin by recalling the definition of the bipolar
filtration~$\{\cT_n\}$.  In this paper, manifolds and submanifolds are
always assumed to be compact, oriented and smooth.  For a knot~$K$,
denote by $M(K)$ the zero-surgery manifold.  Denote by $G^{(n)}$ the
$n$th derived subgroup of a group $G$, which is defined by $G^{(0)}=G$
and $G^{(n+1)}=[G^{(n)},G^{(n)}]$.

\begin{definition}[{\cite[Definition~5.1]{Cochran-Harvey-Horn:2012-1}}]
  \label{definition:positivity-negativity}
  A knot $K$ in $S^3$ is \emph{$n$-negative} if $M(K)$ bounds a
  connected 4-manifold $V$ satisfying the following.
  \begin{enumerate}
  \item The inclusion induces an isomorphism $H_1(M(K))\to H_1(V)$ and
    a meridian of $K$ normally generates~$\pi_1(V)$.
  \item There is a basis for $H_2(V)$ which consists of the classes of
    closed connected surfaces $\{S_i\}$, disjointly embedded in $V$,
    with self-intersection number $S_i\cdot S_i=-1$ (or equivalently,
    with normal bundle with Euler class~$-1$).
  \item For each $i$, the image of $\pi_1(S_i)$ lies
    in~$\pi_1(V)^{(n)}$.
  \end{enumerate}
  The above 4-manifold $V$ is called an \emph{$n$-negaton bounded by
    $M(K)$}.

  An \emph{$n$-positive knot} and an \emph{$n$-positon} are defined by
  replacing the self-intersection condition by $S_i\cdot S_i=+1$.
  A knot is \emph{$n$-bipolar} if it is $n$-positive and $n$-negative.
\end{definition}

Recall that $\cT$ denotes the smooth knot concordance group of
topologically slice knots.  The group operation is connected sum.  For
an integer $n\geq 0$, let $\cT_n$ be the subset of $\cT$ consisting of
the concordance classes of $n$-bipolar knots.  In
\cite{Cochran-Harvey-Horn:2012-1}, it was shown that $\cT_n$ is a
subgroup of~$\cT$.  It is straightforward that
$\cT_{n+1}\subset \cT_n$.

\begin{definition}[{\cite[Definition~2.6]{Cochran-Harvey-Horn:2012-1}}]
  The descending filtration
  \[
    \cT \supset \cT_0 \supset \cT_1 \supset \cdots \supset \cT_n
    \supset\cdots \supset \{0\}
  \]
  is called the \emph{bipolar filtration} of~$\cT$.
\end{definition}

\subsection{Construction of examples}
\label{subsection:construction-of-examples}

Fix an integer $n\ge 2$.  In this section, we construct a sequence of
topologically slice $n$-bipolar knots $K_i$ ($i=1,2,\ldots$) by using
iterated satellite operations.  They will be proven to be
linearly independent in the quotient $\cT_n/\cT_{n+1}$ in later
sections.

We use the following notations for satellite operations.  For a
knot $J$ in $S^3$, denote the exterior by~$E_J$.  Suppose $J$ and $P$
are knots in $S^3$ and $\eta$ is a knot in $S^3\sm P$ which is
unknotted in~$S^3$.  Take the union $E_{\eta} \cup_\partial E_J$,
where the boundaries are attached along an orientation reversing
homeomorphism identifying a zero linking longitude of $\eta$ with a
meridian of $J$ and a meridian of $\eta$ with a zero linking longitude
of~$J$.  Let $P(\eta,J)$ be the image of $P$ in
$E_{\eta} \subset E_{\eta} \cup_\partial E_J \cong S^3$.  This is a
satellite knot with pattern $P$ and companion~$J$.

Our examples are of the following form.  Let $R$ be the knot
$9_{46}$ and $\alpha_J$, $\alpha_D$ be the curves shown in
Figure~\ref{figure:base-seed-knot}.  Denote by $R(J,D)$ the satellite
knot $(R(\alpha_J,J))(\alpha_D,D)$, which is shown in
Figure~\ref{figure:example-knot} in the introduction.  The following
observation will be useful in later parts: for the trivial knot $U$,
both $R(U,D)$ and $R(J,U)$ are slice.  A movie picture of a slice disk
in $D^4$, for instance for $R(U,D)$, is obtained by cutting the
1-handle of the obvious genus one Seifert surface along which $D$ is
tied.

\begin{figure}[H]
  \includestandalone{base-seed-knot}
  \caption{The knot $R=9_{46}$ and the curves $\alpha_J$ and $\alpha_D$}
  \label{figure:base-seed-knot}
\end{figure}

Let $D$ be the untwisted positive Whitehead double of the right-handed
trefoil.  This choice will be fixed throughout this paper.

On the other hand, in place of $J$, we will use knots $J^i_{n-1}$
($i=1,2,\ldots$) described below.  (Recall that $n\ge 2$ is fixed.)
We will use the following notation.  For a knot $J$, let
$\sigma_J(\omega)$
be the Levine-Tristram function defined for $\omega\in S^1$.  For a
positive integer $d$, denote the average of the evaluations of
$\sigma_J$ at the $d$th roots of unity by
\[
  \rho(J,\Z_d) := \frac1d \sum_{k=0}^{d-1} \sigma_J(e^{2\pi k\sqrt{-1}/d}).
\]

We start by choosing a knot $J^i_0$ and a prime $p_i$ for each
$i=1,2,\ldots$ satisfying the following:

\begin{enumerate}[label=({J\arabic*})]
\item\label{item:negativity-of-J_0} For each $i$, $J^i_0$ is
  $0$-negative.
\item\label{item:signature-of-J_0-large-enough} For each $i$,
  $|\rho(J^i_0,\Z_{p_i})| > 69\,713\,280\cdot (6n+90)$.
\item\label{item:independence-of-signature-of-J_0} For $i<j$,
  $\rho(J^j_0,\Z_{p_i})=0$.
\end{enumerate}

An explicit construction of a sequence $\{(J^i_0, p_i)\}$ satisfying
\ref{item:negativity-of-J_0}, \ref{item:signature-of-J_0-large-enough}
and \ref{item:independence-of-signature-of-J_0} will be given in
Section~\ref{section:signature-realization-by-negative-knots}\@.  In
this section, we will use \ref{item:negativity-of-J_0} only.
The other conditions \ref{item:signature-of-J_0-large-enough} and
\ref{item:independence-of-signature-of-J_0} will be used in
Section~\ref{section:using-cheeger-gromov}.

For $k=0,1,\ldots,n-2$, define $J^i_{k+1}:=P_k(\eta_k,J^i_k)$, where
$P_k$ is the stevedore knot and $\eta_k$ is the curve shown in
Figure~\ref{figure:stevedore-pattern}.  Although $(P_k,\eta_k)$
remains the same as $k$ varies, we will keep the index $k$ in the
notation since it will be useful to distinguish the occurrences in
distinct stages.

\begin{figure}[H]
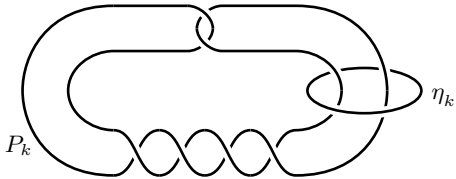

  \includestandalone{stevedore-pattern}
  \caption{The stevedore pattern $(P_k,\eta_k)$.}
  \label{figure:stevedore-pattern}
\end{figure}

Let $K_i$ be the satellite knot~$R(J^i_{n-1},D)$.  Each $K_i$ is
topologically slice since $D$ is topologically slice by the work of
Freedman~\cite{Freedman:1984-1}.  By the following lemma, $K_i$ lies
in~$\cT_n$.

\begin{lemma}
  \label{lemma:bipolarity-of-examples}
  Under the assumption that $J^i_0$ is $0$-negative, the knot $K_i$ is
  $n$-negative.  Also, $K_i$ is $k$-positive for every~$k$.
\end{lemma}

\begin{proof}
  We will use the following two facts. (i) If $P$ is slice, $J$ is
  $n$-negative and $[\eta]\in \pi_1(S^3\sm P)^{(k)}$, then $P(\eta,J)$
  is
  $(n+k)$-negative~\cite[Proposition~3.3]{Cochran-Harvey-Horn:2012-1}.
  This holds when ``negative'' is replaced by ``positive'' or by
  ``bipolar.''  (ii) A knot is $0$-positive if it can be changed to a
  trivial knot by changing positive crossings to negative crossings
  \cite[Proposition~3.1]{Cochran-Harvey-Horn:2012-1},
  \cite[Lemma~3.4]{Cochran-Lickorish:1986-1}.
 
  In our case, observe that the stevedore knot $P_k$ is slice and
  $[\eta_k]$ lies in $\pi_1(S^3\sm P_k)^{(1)}$.  Since $J^i_0$ is
  $0$-negative, $J^i_k$ is $k$-negative for $k=0,1,\ldots$ by
  induction using~(i).  Since $R(U,D)$ is slice and
  $[\alpha_J] \in \pi_1(S^3\sm R(U,D))^{(1)}$, it follows that
  $K_i = (R(U,D))(\alpha_J,J^i_{n-1})$ is $n$-negative once again
  by~(i).

  Let $T$ be the right-handed trefoil.  It is $0$-positive by~(ii).
  The knot $D$ can be viewed as a satellite knot $\Wh(\eta,T)$ where
  $(\Wh,\eta)$ is the pattern shown in
  Figure~\ref{figure:whitehead-pattern}.  Since $[\eta]$ is trivial in
  $\pi_1(S^3\sm\Wh)=\Z$, it follows that $D$ is $k$-positive for all
  $k$ by~(i). Therefore, by~(i), $K_i=(R(J_{n-1}^i,U))(\alpha_D,D)$
  is $k$-positive for all $k$, since $R(J_{n-1}^i,U)$ is slice.
\end{proof}

\begin{figure}[H]
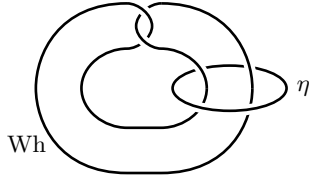

  \includestandalone{whitehead-pattern}
  \caption{The Whitehead pattern $(\Wh,\eta)$.}
  \label{figure:whitehead-pattern}
\end{figure}

\subsection{A negaton from a linear combination and metabolizers}
\label{subsection:first-step-of-proof}

Suppose that a nontrivial finite linear combination
$K=\#_{i=1}^r a_i K_i$ ($a_i\in \Z$) of the knots $K_i$ is
$(n+1)$-bipolar.  By eliminating terms with $a_i=0$ and by replacing
$K$ by $-K$ if necessary, we may assume that $a_1\ge 1$ and $a_i\ne 0$
for each~$i$.  Our strategy is to derive a contradiction by
investigating consequences \emph{on the first knot $K_1$} which are
implied by the hypothesis on the linear combination~$K$.

For this purpose, we will construct a specific $n$-negaton for $K_1$,
from a given $(n+1)$-negation for $K$, by attaching additional
``negative'' pieces.  (Principle: negative + negative = negative.)  It
will be guided by the observation that $K_1$ is concordant to the
connected sum of $K$, $-(a_1-1)K_1$ and $-a_iK_i$ ($i>1$), where
the summands added to $K$ are $n$-negative regardless of the sign of
$a_i$, by Lemma~\ref{lemma:bipolarity-of-examples}.

The actual construction proceeds as follows.  Let $V^-$ be an
$(n+1)$-negaton bounded by~$M(K)$.  For each $i$, if $a_i>0$, choose
an $n$-negaton bounded by $-M(K_i)=M(-K_i)$ by invoking
Lemma~\ref{lemma:bipolarity-of-examples} and call it~$Z^-_i$.  If
$a_i<0$, choose an $n$-negaton bounded by $M(K_i)$ again by using
Lemma~\ref{lemma:bipolarity-of-examples} and call it~$Z^-_i$.  Indeed,
we will use a specific choice of $Z_i^-$ later in
Section~\ref{subsection:estimate-cheeger-gromov-invariant} (see
Lemma~\ref{lemma:specific-choice-of-n-negaton}), but for now it
suffices to assume that $Z^-_i$ is just an $n$-negaton for~$\pm K_i$.
Recall that there is a standard cobordism, which we call $C$, bounded
by the union of $\partial_-C := -M(K)$ and
$\partial_+C := \bigsqcup_{i=1}^r a_i M(K_i)$.  It is obtained by
attaching, to $\bigsqcup_{i=1}^r a_i M(K_i) \times I$, $(N-1)$
1-handles that makes it connected and attaching $(N-1)$ 2-handles that
makes meridians of the involved knots parallel, where
$N=\sum_{i=1}^r |a_i|$.  See, for instance,
\cite[p.~113]{Cochran-Orr-Teichner:2002-1} for a detailed discussion
on~$C$.  Define
\begin{equation}
  \label{equation:definition-of-negaton}
  X^- := V^- \cupover{\partial_-C} C \cupover{\partial_+C}
  \Big((a_1-1)Z^-_1 \sqcup \bigsqcup_{i>1} |a_i| Z^-_i\Big).
\end{equation}
See the schematic diagram in Figure~\ref{figure:negaton-diagram}.

\begin{figure}[H]
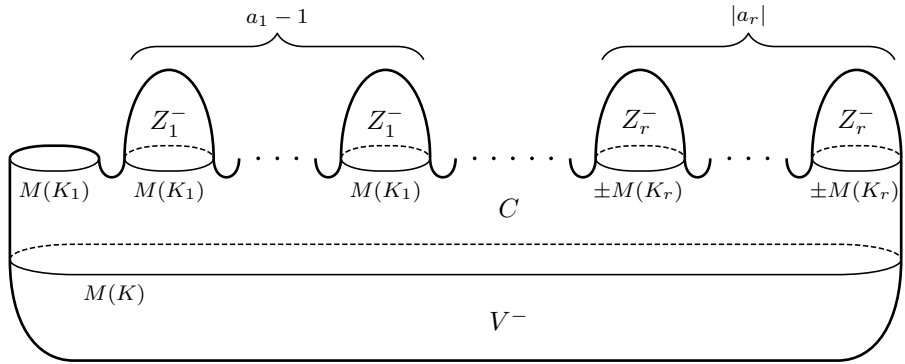

  \includestandalone{negaton-diagram}
  \caption{A schematic diagram of the negaton~$X^-$.  Here the sign of
    $M(K_i)$ is equal to the sign of~$a_i$.}
  \label{figure:negaton-diagram}
\end{figure}

\begin{lemma}\label{lemma:X^-isnegaton}
  The $4$-manifold $X^-$ is an $n$-negaton bounded by~$M(K_1)$.
\end{lemma}

\begin{proof}
  It is known that $H_1(C) \cong \Z$ and
  $H_2(C)\cong \bigoplus_i H_2(M(K_i))^{|a_i|}$ where the second
  isomorphism is induced by the inclusions (for instance see
  \cite[p.~113]{Cochran-Orr-Teichner:2002-1}).  Also, a meridian of
  any one of $K, K_1,\ldots,K_r$ normally generates $\pi_1(C)$ (and
  hence generates $H_1(C)$), because of the 2-handle attachments in
  the construction of~$C$.  Using this and
  Definition~\ref{definition:positivity-negativity}~(1) for the
  negatons $Z^-_i$ and $V^-$, the following is shown by a
  Mayer-Vietoris argument:
  \begin{align*}
    H_1(X^-) &\cong H_1(M(K_1)),
    \\
    H_2(X^-) &\cong H_2(V^-) \oplus H_2(Z^-_1)^{a_1-1} \oplus
    \bigg( \bigoplus_{i>1} H_2(Z^-_i)^{|a_i|} \bigg)
  \end{align*}
  where the isomorphisms are inclusion-induced.  From the $H_2$
  computation, it follows that
  Definition~\ref{definition:positivity-negativity}~(2) and (3) are
  satisfied for~$X^-$.  Since $\pi_1(C)$, $\pi_1(V^-)$ and
  $\pi_1(Z^-_i)$ are normally generated by meridians of $K_1$, $K$ and
  $K_i$ respectively, it follows that $\pi_1(X^-)$ is normally
  generated by a meridian of~$K_1$.  By this and the above $H_1$
  computation, Definition~\ref{definition:positivity-negativity}~(1)
  is satisfied.
\end{proof}

Recall that the Blanchfield form
\[
  \Bl\colon H_1(M(J);\Q[t^{\pm1}])\times H_1(M(J);\Q[t^{\pm1}])
  \to \Q(t)/\Q[t^{\pm1}]
\]
is defined on the (rational) Alexander module $H_1(M(J);\Q[t^{\pm1}])$
of a knot~$J$, and that a submodule $P$ of $H_1(M(J);\Q[t^{\pm1}])$ is
called a \emph{metabolizer} if $P=P^\perp$, where
\[
  P^\perp:=\{x\in H_1(M(J);\Q[t^{\pm1}]) \mid \Bl(x,P)=0\}.
\]

\begin{lemma}[A special case of
  {\cite[Theorem~5.8]{Cochran-Harvey-Horn:2012-1}}]
  \label{lemma:negaton-metabolizer}
  Let $V$ be a either $1$-negaton or $1$-positon bounded by $M(J)$,
  and let $P$ be the kernel of the inclusion-induced homomorphism
  $H_1(M(J);\Q[t^{\pm1}]) \to H_1(V;\Q[t^{\pm1}])$ on the Alexander
  modules.  Then $P$ is a metabolizer of the Blanchfield form.
\end{lemma}

Returning to our case, let
\[
  P:= \Ker \big\{ H_1(M(K_1);\Q[t^{\pm1}]) \to H_1(X^-;\Q[t^{\pm1}])
  \big\}.
\] 

We need the following facts, which can be verified by a routine
computation, for instance using the Seifert matrix $\sbmatrix{0 & 2\\
  1 & 0}$ of~$K_1$.  Regard (zero linking parallels of) the curves
$\alpha_J$ and $\alpha_D$ in Figure~\ref{figure:base-seed-knot} as
curves in the zero surgery manifold of $K_1=R(J^1_{n-1},D)$, and
denote them by $\alpha_J,\alpha_D \subset M(K_1)$, for brevity.  Then,
$H_1(M(K_1);\Q[t^{\pm1}])$ is the internal direct sum of two cyclic
submodules $\langle\alpha_J\rangle$ and $\langle\alpha_D\rangle$
generated by the classes of $\alpha_J$ and $\alpha_D$ respectively,
and in addition, $\langle\alpha_J\rangle \cong \Q[t^{\pm1}]/(t-2)$ and
$\langle\alpha_D\rangle \cong \Q[t^{\pm1}]/(2t-1)$.  The Blanchfield
form of $K_1$ is given by
$\Bl(\alpha_J,\alpha_J) = \Bl(\alpha_D,\alpha_D) = 0$,
$\Bl(\alpha_J,\alpha_D) = \frac{t-1}{1-2t}$.  For later use in
Section~\ref{section:using-d-invariant}, we remark that the same
conclusion holds when $\Z$ is used as coefficients in place of~$\Q$.

From the above paragraph, it follows that $P$ is equal to either
$\langle\alpha_D\rangle$ or~$\langle\alpha_J\rangle$, since $P$ is a
metabolizer by Lemma~\ref{lemma:negaton-metabolizer}.

\begin{case-named}[Case 1: $P=\langle\alpha_D\rangle$] In this case,
  we will use the Cheeger-Gromov $L^2$ $\rho$-invariants to derive a
  contradiction.  The proof is given in
  Section~\ref{section:using-cheeger-gromov}.
\end{case-named}

\begin{case-named}[Case 2: $P=\langle\alpha_J\rangle$]
  In this case, to derive a contradiction, we will use the Heegaard
  Floer correction term $d$-invariants of infinitely many branched
  covers of~$K_1$.  The proof is given in
  Section~\ref{section:using-d-invariant}.
\end{case-named}

The proof of Theorem~\ref{theorem:main-nontriviality-result} will be
finished by completing the above two cases.

\section{Case 1: Use of Cheeger-Gromov invariants}
\label{section:using-cheeger-gromov}

The goal of this section is to reach a contradiction in Case~1
described above.  Suppose $P=\langle\alpha_D\rangle$ throughout this
section.

Our key ingredient is the amenable signature theorem developed in
\cite{Cha-Orr:2009-1,Cha:2010-1}.  We use it to extract obstructions,
from Cheeger-Gromov invariants over locally $p$-indicable amenable
groups.  For this purpose, we begin by converting the negatons
described in Section~\ref{subsection:first-step-of-proof} to
4-manifolds called \emph{integral solutions} in~\cite{Cha:2010-1}.
After that, we analyze the behavior of a commutator series of the
fundamental group, and investigate Cheeger-Gromov invariants over the
associated quotient, by applying the ideas and methods used
in~\cite[Sections 4 and~5]{Cha:2010-1}.  We remark that this type of
technique is strongly influenced by earlier work of Cochran, Harvey
and Leidy~\cite{Cochran-Harvey-Leidy:2009-1}.

\subsection{A 4-manifold and analysis of mixed-type commutator series}
\label{subsection:manifold-construction-and-mixed-type-commutator-analysis}

Recall from \eqref{equation:definition-of-negaton} that $X^-$ is the
union of $V^-$, $C$, $(a_1-1)Z^-_1$ and $|a_i| Z^-_i$ ($i>1$).  Let
$V^0 = V^- \csum (b_2(V^-) \C P^2)$ and
$Z^0_i = Z^-_i \csum (b_2(Z^-_i) \C P^2)$.  Then $V^0$ is an integral
$(n+1)$-solution in the sense of~\cite[Definition~3.1]{Cha:2010-1},
and each $Z^0_i$ is an integral $n$-solution too, by
\cite[Proposition~5.5]{Cochran-Harvey-Horn:2012-1}.  Since we do not
directly use the definition of an integral $n$-solution, we do not
spell it out but we will state some properties when we need to use
them.  We remark that the definition of an integral $n$-solution does
not require a spin structure (and so there is no condition on the self-intersection $\mu$), in contrast to the definition of an 
$n$-solution defined in~\cite{Cochran-Orr-Teichner:1999-1}.

Let
\begin{equation}
  \label{equation:definition-of-solution}
  X^0 := V^0 \cupover{\partial_-C} C \cupover{\partial_+C}
  \Big((a_1-1)Z^0_1 \sqcup \bigsqcup_{i>1} |a_i| Z^0_i\Big).
\end{equation}
The manifold $X^0$ is bounded by $M(K_1)$.  From the hypothesis that
$P=\langle\alpha_D\rangle$, it follows that the kernel of
$H_1(M(K_1);\Q[t^{\pm1}])\to H_1(X^0;\Q[t^{\pm1}])$ is still equal to
$\langle\alpha_D\rangle$, since $\pi_1(X^0)\cong \pi_1(X^-)$.  In
fact, as done in \cite{Cochran-Harvey-Horn:2012-1}, this leads us to a
proof of Lemma~\ref{lemma:negaton-metabolizer}, since it is known that
if $W$ is an integral $n$-solution with $n\ge 1$ for a knot $J$, then
the kernel of $H_1(M(J);\Q[t^{\pm1}]) \to H_1(W;\Q[t^{\pm1}])$ is a
metabolizer~\cite[Theorem~4.4]{Cochran-Orr-Teichner:1999-1}.

Attach more pieces to $X^0$ as follows.  For the satellite knot
$J^1_{k+1} = P_k(\eta_k, J^1_k)$, there is a standard cobordism, say
$E_k$, from $M(J^1_{k+1})$ to $M(J^1_k)\sqcup M(P_k)$ for
$k=0,\ldots,n-2$: take the union of $M(J^1_k)\times I$ and
$M(P_k)\times I$, and identify the solid torus
$\overline{M(J^1_k)\sm E(J^1_k)}\times 1$ with a tubular neighborhood
of
$\eta_k\subset P_k\times
1$~\cite[p.~1429]{Cochran-Harvey-Leidy:2009-1}.  The same construction
gives a standard cobordism $E_{n-1}$ from $M(K_1)$ to
$M(J^1_{n-1})\sqcup M(R(U,D))$.  Define
\[
  X := X^0 \cupover{M(K_1)} E_{n-1} \cupover{M(J^1_{n-1})} E_{n-2}
  \cupover{M(J^1_{n-2})} \cdots \cupover{M(J^1_{1})} E_{0}.
\]
See the schematic diagram in
Figure~\ref{figure:higher-cobordism-diagram}.

\begin{figure}[t]
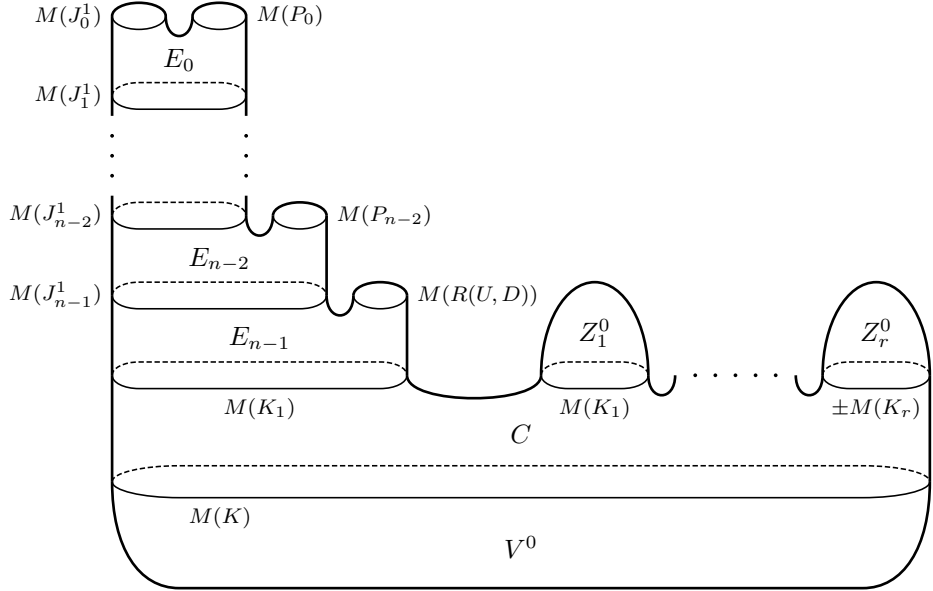

  \includestandalone{higher-cobordism-diagram}
  \caption{A schematic diagram of the 4-manifold~$X$.}
  \label{figure:higher-cobordism-diagram}
\end{figure}

We will compute the $\rhot$-invariant of $\partial X$, which is
associated to groups obtained from a certain mixed-coefficient
commutator series construction, as first done in~\cite{Cha:2010-1}.
The series is defined as follows.  Let $p=p_1$, which is the prime
associated to~$J^1_0$ (see the conditions
\ref{item:negativity-of-J_0}, \ref{item:signature-of-J_0-large-enough}
and~\ref{item:independence-of-signature-of-J_0} in
Section~\ref{subsection:construction-of-examples}).  Let $R_i=\Q$ for
$i=0,\ldots,n-1$, and let $R_n=\Z_p$.  For a group $G$, define
$\cP^0 G:=G$ and
\[
  \cP^{k+1}G = \Ker\Big\{ \cP^{k}G \to
    \frac{\cP^{k}G}{[\cP^{k}G,\cP^{k}G]} \to
    \frac{\cP^{k}G}{[\cP^{k}G,\cP^{k}G]} \otimesover{\Z} R_k \Big\}
\]
for $k=0,\ldots,n$ inductively.  For the use in
\eqref{equation:amenable-signature-vanishing} (see
Section~\ref{subsection:estimate-cheeger-gromov-invariant}), we note
that the quotient $\Gamma=G/\cP^{n+1} G$ satisfies
$\Gamma^{(n+1)}=\{1\}$, and that $\Gamma$ is amenable and locally
$p$-indicable by~\cite[Lemma~6.8]{Cha-Orr:2009-1}.  Here, a group is
\emph{amenable} if it admits a finitely additive invariant mean, and
is \emph{locally $p$-indicable} if each nontrivial finitely generated
subgroup admits an epimorphism onto~$\Z_p$.  (In this paper we will
not use these definitions directly.)

\begin{remark}
  \label{remark:strebel-and-local-indicability}
  In place of local $p$-indicability, Strebel's class $D(\Z_p)$
  \cite{Strebel:1974-1} was used in statements in
  \cite{Cha-Orr:2009-1} and subsequent papers.  Indeed, a group is
  locally $p$-indicable if and only if it lies in $D(\Z_p)$, due
  to~\cite{Howie-Schneebeli:1983-1}.
\end{remark}

Let $\phi\colon \pi_1(X)\to \pi_1(X)/\cP^{n+1} \pi_1(X)$ be the quotient
homomorphism.  Let $\mu_{J^1_0}$ be the meridian of $J^1_0$, and
regard it as a curve in $M(J^1_0) \subset \partial X \subset X$.

The following is the key part which makes use of the defining
condition of Case~1.

\begin{lemma}
  \label{lemma:nontriviality-of-representation}
  Under the assumption of Case~1 that $P=\langle \alpha_D \rangle$,
  the element $\phi(\mu_{J^1_0})$ is nontrivial and lies in the
  subgroup $\cP^{n} \pi_1(X) / \cP^{n+1} \pi_1(X)$.  In addition,
  $\phi(\mu_{J^1_0})$ has order~$p$.
\end{lemma}

\begin{proof}[Outline of the proof]
  
  Except the nontriviality of $\phi(\mu_{J^1_0})$, the assertions in
  Lemma~\ref{lemma:nontriviality-of-representation} are shown
  straightforwardly.  Indeed, reverse induction on
  $k=n-1,n-2,\ldots,0$ shows that $\mu_{J^1_k}$ lies in
  $\pi_1(X)^{(n-k)} \subset \cP^{n-k}\pi_1(X)$; use that $\mu_{J^1_k}$
  is parallel to the satellite curve $\eta_k$ which lies in
  $\pi_1(M(P_k))^{(1)}$, and that the meridian of $P_k$, which
  normally generates $\pi_1(M(P_k))$, is homotopic to
  $\mu_{J^1_{k+1}}$ in~$X$.  Also, the order assertion in
  Lemma~\ref{lemma:nontriviality-of-representation} follows from the
  nontriviality, since $\cP^n \pi_1(X) / \cP^{n+1} \pi_1(X)$ is isomorphic
  to a direct sum of (possibly infinitely many) copies of $\Z_p$, by
  the definition of the mixed coefficient commutator series.

  The nontriviality of $\phi(\mu_{J^1_0})$ in
  Lemma~\ref{lemma:nontriviality-of-representation} is proven by
  exactly the same argument as the proof of Theorem~4.14
  in~\cite{Cha:2010-1}, which is given in Section~5
  of~\cite{Cha:2010-1}.  So, instead of presenting full details, we
  will discuss the key difference in our case, focusing on the role of
  the hypothesis $P=\langle\alpha_D\rangle$ in Case 1.

  Theorem~4.14 in~\cite{Cha:2010-1} gives the desired nontriviality
  for another 4-manifold, denoted by $W_0$ in~\cite{Cha:2010-1},
  instead of our~$X$.  The manifold $W_0$ in~\cite{Cha:2010-1} is
  constructed in the exactly same way as $X$, but using a different
  knot (indeed the stevedore knot) in place of our $R(U,D)$, which is
  the pattern used to produce $K_1$ from the companion~$J^1_{n-1}$.
  This different choice in~\cite{Cha:2010-1} automatically provides
  the property that the satellite curve used to produce $K_1$, which
  is the analogue of $\alpha_J$ in Figure~\ref{figure:base-seed-knot}
  in our case, does not lie in the kernel $P$ of the homomorphism
  $H_1(M(K_1);\Q[t^{\pm1}]) \to H_1(X^0;\Q[t^{\pm1}])$.  (This
  property is used in lines 1--5 on page 4801 in~\cite{Cha:2010-1}.)
  In our case, $\alpha_J\not\in P$ is guaranteed by the hypothesis
  $P=\langle \alpha_D \rangle$ of Case~1.  This enables us to carry
  out all the arguments as in the proof of Theorem~4.14 given in
  \cite[Section~5]{Cha:2010-1}, to prove the nontriviality
  of~$\phi(\mu_{J^1_0})$.
\end{proof}

\subsection{Estimating Cheeger-Gromov invariants}
\label{subsection:estimate-cheeger-gromov-invariant}

Now we estimate the Cheeger-Gromov invariant over the mixed-type
commutator series quotient~$\pi_1(X)/\cP^{n+1}\pi_1(X)$.  For the
reader's convenience, we describe some known key results on the
Cheeger-Gromov invariants employed in our argument.

\begin{definition-named}[$L^2$-signature defect interpretation]
  Suppose $M$ is a 3-manifold bounding a 4-manifold $W$ and
  $\phi\colon \pi_1(M)\to G$ is a homomorphism factoring
  through~$\pi_1(W)$.  Then the Cheeger-Gromov invariant
  $\rhot(M,\phi)$ is equal to the $L^2$-signature defect of $W$ over
  $G$, which we denote by~$\osigmat_G(W)$.  That is,
  \begin{equation}
    \label{equation:l2sign-interpretation}
    \rhot(M,\phi) = \osigmat_G(W) := \sign^{(2)}_G(W) -\sign(W)
  \end{equation}
  where $\sign^{(2)}_G(W)$ is the $L^2$-signature of the intersection
  form
  \[
    H_2(W;\N G) \times H_2(W;\N G) \to \N G
  \]
  with the group von Neumann algebra $\N G$ as coefficients, and
  $\sign(W)$ is the ordinary signature of~$W$.  For more about this,
  see, for instance, \cite{Chang-Weinberger:2003-1},
  \cite[Section~2]{Cochran-Teichner:2003-1},
  \cite[Section~3]{Harvey:2006-1}, \cite[Section~2.1]{Cha:2014-1}.
\end{definition-named}

\begin{definition-named}[Quantitative universal bound (a special case)]
  If $K$ has a planar diagram with $c$ crossings, then for every
  homomorphism $\phi$ of~$\pi_1(M(K))$,
  \begin{equation}
    \label{equation:quantitative-bound}
    |\rhot(M(K),\phi)| \le 69\,713\,280 \cdot c.
  \end{equation}
  This is a special case of~\cite[Theorem~1.9]{Cha:2014-1}. 
\end{definition-named}

\begin{definition-named}[Amenable signature theorem (a special case)]
  Suppose $W$ is an integral $(n.5)$-solution bounded by $M(J)$, $G$
  is a locally $p$-indicable amenable group satisfying
  $G^{(n+1)}=\{1\}$, and $\phi\colon \pi_1(M(J))\to G$ is a
  homomorphism which factors through $\pi_1(W)$ and takes a meridian
  to an infinite order element.  Then
  \begin{equation}
    \label{equation:amenable-signature-vanishing}
    \rhot(M(J),\phi)=\osigmat_G(W)=0.  
  \end{equation}
  This is a special case of~\cite[Theorem~3.2]{Cha:2010-1}, whose
  proof relies on~\cite{Cha-Orr:2009-1}.  (See also
  Remark~\ref{remark:strebel-and-local-indicability}.)  We remark that
  the amenable signature theorem is a generalization of a major result
  in~\cite{Cochran-Orr-Teichner:1999-1}.
\end{definition-named}

Also, the following explicit computation is useful for our purpose.

\begin{definition-named}[Computation for knots over a finite cyclic
  group]
  If $\phi\colon \pi_1(M(J)) \to G$ is a homomorphism with finite
  cyclic image of order $d$, then
  \begin{equation}
    \label{equation:abelian-computation}
    \rhot(M(J),\phi) = \rho(J,\Z_d) := \frac1d \sum_{k=0}^{d-1}
    \sigma_{J}(e^{2\pi k\sqrt{-1}/d}).
  \end{equation}
  A proof can be found in \cite[Corollary~4.3]{Friedl:2003-5},
  \cite[Lemma~8.7]{Cha-Orr:2009-1}.
\end{definition-named}
  
Returning to our case, let
$\phi\colon \pi_1(X) \to G:=\pi_1(X)/\cP^{n+1}\pi_1(X)$ be the
projection.  For brevity, for a subspace $A$ of $X$, denote the
restriction of $\phi$ on $\pi_1(A)$ by~$\phi$.  By applying the
$L^2$-signature defect
interpretation~\eqref{equation:l2sign-interpretation} and the Novikov
additivity to the 4-manifold $X$, we obtain the following:

\begin{equation}
  \label{equation:L2-signature-of-X}
  \begin{aligned}
    &\rhot(M(J^1_0),\phi) + \sum_{k=0}^{n-2} \rhot(M(P_k),\phi) +
    \rhot(M(R(U,D)),\phi)
    \\
    &\qquad = \osigmat_G(X) = \osigmat_G(V^0) + \osigmat_G(C) +
    \sum_{k=0}^{n-1} \osigmat_G(E_k) + \sum_{i,j}
    \osigmat_G(Z^0_{i,j}),
  \end{aligned}
\end{equation}
where $Z^0_{i,j}$ designates the $j$th copy of $Z^0_i$ used in the
construction of $X$.  (We use this notation just because each copy of
$Z^0_i$ may contribute different $L^2$-signature defect.)

Since $\phi$ restricted on $\pi_1(M(J^1_0))$ is onto a subgroup
isomorphic to $\Z_p$ (see Lemma~\ref{lemma:nontriviality-of-representation}),
\begin{equation}
  \label{equation:computation-of-rho-for-J0}
  \rhot(M(J^1_0),\phi) = \rho(J^1_0,\Z_p) 
\end{equation}
by~\eqref{equation:abelian-computation}.  By the quantitative
universal bound~\eqref{equation:quantitative-bound},
\begin{equation}
  \label{equation:universal-bound-computation-1}
  \big|\rhot\big(M(P_k),\phi\big)\big| \le 6\cdot 69\,713\,280 
\end{equation}
since the stevedore knot $P_k$ has 6 crossings.  Similarly,
by~\eqref{equation:quantitative-bound},
\begin{equation}
  \label{equation:universal-bound-computation-2}
  \big|\rhot\big(M(R(U,D)),\phi\big)\big| \le 96\cdot 69\,713\,280 
\end{equation}
since $R(U,D)$ has a diagram with 96 crossings.  Recall that $V^0$ is
an integral $(n+1)$-solution, $G$ is locally $p$-indicable and
$G^{(n+1)}=\{1\}$.  Since the meridian of $K$ generates
$H_1(X)\cong \Z$ onto which $G$ surjects, the meridian of $K$ has
infinite order in~$G$.  So we have
\begin{equation}
  \label{equation:L2-signature-for-V0}
  \osigmat_G(V^0) = \rhot(M(K),\phi)=0  
\end{equation}
by the amenable signature
theorem~\eqref{equation:amenable-signature-vanishing}.

Recall that $Z^0_{i,j}= Z^-_{i,j} \# (b_2(Z^-_{i,j})\C P^2)$, where
$Z^-_{i,j}$ has been assumed to be an arbitrary $n$-negaton for
$\pm K_i$.  To control $\osigmat_G(Z^0_{i,j})$, we use a specific
choice of~$Z^0_{i,j}$ described below.

\begin{lemma}
  \label{lemma:specific-choice-of-n-negaton}
  Under our assumption that $J^i_0$ is a $0$-negative knot, there
  exists an $n$-negaton $Z^-_{i,j}$ for $\pm K_i$ satisfying that
  $\osigmat_G(Z^0_{i,j})$ is equal to either zero or
  $\rho(J^i_0,\Z_p)$.
\end{lemma}

\begin{proof}
  Recall from the definition that
  \[
    K_i=R(P_{n-2}(\eta_{n-2},\ldots,P_{0}(\eta_{0},J^i_0) \cdots ),
    D).
  \]
  Let $Q$ be the knot obtained by replacing $J^i_0$ in this expression
  by the trivial knot~$U$.  Since $U$ and all the $P_k$ are slice, $Q$
  is slice.  We can view $\eta_0$ as a curve in $S^3 \sm Q$, and $K_i$
  can be written as $K_i=Q(\eta_0,J^i_0)$.  Since
  $[\eta_k] \in \pi_1(S^3 \sm P_k)^{(1)}$ for each $k$,
  $[\eta_0] \in \pi_1(S^3\sm Q)^{(n)}$ by induction. 
  
  Choose a slice disk exterior, say $N$, for the slice knot $Q$, and
  choose a $0$-negaton, say $N^-$, for the $0$-negative knot~$J^i_0$.
  Now, as our $Z^-_{i,j}$, take the union of $N$ and $N^-$, with a
  tubular neighborhood of $\eta_0 \subset M(Q)=\partial N$ and the
  solid torus $\overline{M(J^i_0)\sm E_{J^i_0}} \subset \partial N^-$
  identified.  In $Z^-_{i,j}$, $\eta_0$ is isotopic to a meridian of
  $J^i_0$, which normally generates $\pi_1(N^-)$ by
  Definition~\ref{definition:positivity-negativity}~(1) since $N^-$ is
  a 0-negaton.  Since $[\eta_0]\in \pi_1(S^3\sm Q)^{(n)}$,
  $\pi_1(N^-)$ maps to $\pi_1(Z^-_{i,j})^{(n)}$, and from this it
  follows that $Z^-_{i,j}$ is an $n$-negaton.

  To obtain the integral $n$-solution $Z^0_{i,j}$, first let $N^0$ be
  the connected sum of $N^-$ with $b_2(N^-)$ copies of $\C P^2$.  Then
  $N^0$ is an integral $0$-solution for~$J^i_0$, and
  $Z^0_{i,j} = N\cupover{S^1\times D^2} N^0$.  The following
  additivity is known (e.g.,
  see~\cite[Proposition~3.2]{Cochran-Orr-Teichner:2002-1},
  \cite[Proposition~4.4]{Cha:2010-1} and their proofs):
  $\osigmat_G(Z^0_{i,j}) = \osigmat_G(N) + \osigmat_G(N^0)$.  By the
  amenable signature
  theorem~\eqref{equation:amenable-signature-vanishing},
  $\osigmat_G(N)=0$ since $N$ is a slice disk exterior.  By the
  $L^2$-signature defect
  interpretation~\eqref{equation:l2sign-interpretation},
  $\osigmat_G(N^0) = \rhot(M(J^i_0),\phi)$ since
  $\partial N^0=M(J^i_0)$.  Since $[\eta_0]\in \pi_1(S^3\sm Q)^{(n)}$,
  $\pi_1(M(J^i_0))$ maps into $\pi_1(X)^{(n)}\subset \cP^n\pi_1(X)$.
  So the conclusion follows by
  applying~\eqref{equation:abelian-computation}.
\end{proof}

By the property \ref{item:independence-of-signature-of-J_0} in
Section~\ref{subsection:construction-of-examples}, we have
$\rho(J^i_0,\Z_p) = 0$ for $i>1$.  By
Lemma~\ref{lemma:specific-choice-of-n-negaton}, it follows that
\begin{equation}
  \label{equation:vanishing-of-L2-signature-of-Z0}
  \osigmat_G(Z^0_{i,j}) = 0.
\end{equation}
for all $i>1$ and for all~$j$.

Finally, $\osigmat_G(C)=0$ and $\osigmat_G(E_k)=0$ for each $k$, by
\cite[Lemma~2.4]{Cochran-Harvey-Leidy:2009-1}.  From this and
\eqref{equation:L2-signature-of-X},
\eqref{equation:computation-of-rho-for-J0},
\eqref{equation:universal-bound-computation-1},
\eqref{equation:universal-bound-computation-2},
\eqref{equation:L2-signature-for-V0} and
\eqref{equation:vanishing-of-L2-signature-of-Z0}, we obtain the
following:
\[
  r \cdot \big|\rho(J^1_0,\Z_p)\big| \le \bigg| \sum_{k=0}^{n-2}
  \rhot(M(P_k),\phi) \bigg|+ \big|\rhot(M(R(U,D)),\phi)\big| \le
  69\,713\,280\cdot(6n+90)
\]
where $r\ge 1$ is an integer.  But it contradicts the
property~\ref{item:signature-of-J_0-large-enough} of~$J^1_0$.  This
completes the proof for Case~1.

\section{Realization of signature functions by negative knots}
\label{section:signature-realization-by-negative-knots}

In the arguments in Section~\ref{section:using-cheeger-gromov}, the
properties \ref{item:negativity-of-J_0},
\ref{item:signature-of-J_0-large-enough} and
\ref{item:independence-of-signature-of-J_0} stated in
Section~\ref{subsection:construction-of-examples} were among the
essential ingredients used to realize a nontrivial value of the
amenable $\rho$-invariant obstruction.  In this section we describe a
construction of an infinite sequence of knots $J^i_0$ with primes
$p_i$ ($i=1,2,\ldots$) satisfying \ref{item:negativity-of-J_0},
\ref{item:signature-of-J_0-large-enough} and
\ref{item:independence-of-signature-of-J_0}.  For the reader's
convenience, we recall them below.

\begin{enumerate}[label=({J\arabic*})]
\item For each $i$, $J^i_0$ is $0$-negative.
\item For each $i$,
  $|\rho(J^i_0,\Z_{p_i})| > 69\,713\,280\cdot (6n+90)$.
\item For $i<j$, $\rho(J^j_0,\Z_{p_i})=0$.
\end{enumerate}

Here, $\rho(J,\Z_{d})$ is the average of the values of the
Levine-Tristram signature function of $J$ at the $d$th roots of unity
(see Equation~\eqref{equation:abelian-computation}).

We begin with a realization of an arbitrary Alexander polynomial by a
0-negative knot.
\begin{lemma}
  \label{lemma:alexader-polynomial-realization-by-0-negative-knot}
  For every Alexander polynomial $\Delta(t)$ over $\Z$, there is a
  $0$-negative knot whose Alexander polynomial is equal to $\Delta(t)$
  up to multiplication by~$\pm t^k$.
\end{lemma}

\begin{proof}
  Write the given Alexander polynomial as
  \[
    \Delta(t)=a_{2g}(t^g+t^{-g}) + \cdots + a_1(t+t^{-1}) +a_0,
  \]
  with $a_i\in \Z$, $a_0+\sum_{i=1}^g 2a_i=\Delta(1)=-1$.  Invoke a
  classical realization method of Levine~\cite{Levine:1966-1} as
  follows.  Perform $-1$ surgery along the unknotted circle $\alpha$
  in Figure~\ref{figure:levine-realization}, and regard the other
  unknotted circle as a knot $K$ in the result of surgery, which
  is~$S^3$.  Then $K$ has Alexander polynomial~$\Delta(t)$, due
  to~\cite[Proof of Theorem~2]{Levine:1966-1}.

  \begin{figure}[t]
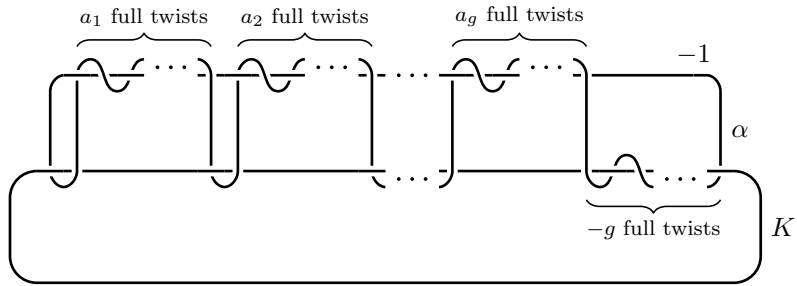

    \includestandalone{levine-realization}
    \caption{Levine's knot with a given Alexander polynomial.}
    \label{figure:levine-realization}
  \end{figure}

  Now, regard Figure~\ref{figure:levine-realization} as a 2-component
  link $\alpha\cup K$ in $S^3$, and let $D$ be the standard slicing
  disk in $D^4$ for the unknotted component~$K$.  Attach a 2-handle to
  the exterior of $D\subset D^4$ along the $-1$ framing of $\alpha$,
  and call the result~$V$.  After the $-1$ surgery, the zero framing
  on $K$ in $S^3$ remains the zero framing, since $\lk(K,\alpha)=0$.
  So $\partial V=M(K)$.  Also, since $\pi_1(D^4\sm D) = \Z$ and
  $\lk(K,\alpha)=0$, we have
  $\pi_1(V) = \pi_1(D^4\sm D)/\langle \alpha \rangle \cong \Z$.
  Choose a surface in the exterior of $D\subset D^4$ bounded by
  $\alpha$, and take the union with the surgery core disk.  This gives
  us a closed surface in $V$, say $S$, which generates
  $H_2(V)\cong \Z$.  Since the surgery framing is $-1$, the self
  intersection of $S$ is~$-1$.  It follows that $V$ is a 0-negaton
  bounded by~$M(K)$.
\end{proof}

Also, we will use the following result of the first named author and
Livingston~\cite{Cha-Livingston:2002-1}.

\begin{lemma}[{\cite[Proof of Theorem 1]{Cha-Livingston:2002-1}}]
  \label{lemma:cha-livingston-realization}
  Suppose $0<\theta_0<\pi$ and $\epsilon>0$.  If $\frac ab$ is a
  rational number sufficiently close to $\cos\theta_0$ with
  sufficiently large $b>0$, then for some
  $\theta_1 \in (\theta_0-\epsilon, \theta_0+\epsilon)$, every knot
  $K$ with Alexander polynomial
  \[
    \Delta_K(t)=bt^2-(2b+2a)t+(4a+2b-1)-(2b+2a)t^{-1}+bt^{-2}
  \]
  has the property that $\sigma_K(e^{\theta\sqrt{-1}})=0$ for
  $0\le\theta<\theta_1$, and $|\sigma_K(e^{\theta\sqrt{-1}})|\ge 2$ for
  $\theta_1<\theta\le \pi$.
\end{lemma}
  
Now, choose increasing odd primes $p_1<p_2<\cdots$.  For each
$i=1,2,\ldots,$ apply Lemma~\ref{lemma:cha-livingston-realization} and
then
Lemma~\ref{lemma:alexader-polynomial-realization-by-0-negative-knot}
to choose a 0-negative knot $J^i_0$ and a real number
$d_i\in (\pi-\frac{\pi}{p_{i-1}}, \pi-\frac{\pi}{p_{i}})$ such that
$|\sigma_{J^i_0}(e^{\theta\sqrt{-1}})| \ge 2$ for
$d_i < \theta \le \pi$, and $|\sigma_{J^i_0}(e^{\theta\sqrt{-1}})|=0$
for $0 \le \theta < d_i$.  (Here, for brevity, let $p_0=2$ so that
$\pi-\frac{\pi}{p_{i-1}}$ is understood as $\frac{\pi}{2}$ for $i=1$.)
Since $\sigma_{J^i_0}(\omega)=\sigma_{J^i_0}(\overline\omega)$, it
follows that
\begin{equation}
  \label{equation:nontriviality-of-signature-average}
  |\rho(J^i_0,\Z_{p_i})| \ge \frac{1}{p_i}
  \big|\sigma_{J^i_0}\big(e^{(\pi-\frac{\pi}{p_i})\sqrt{-1}}\big) +
  \sigma_{J^i_0}\big(e^{(\pi+\frac{\pi}{p_i})\sqrt{-1}}\big)\big|
  \ge \frac{4}{p_i}.
\end{equation}
Furthermore, for $i<j$,
\[
  \rho(J^j_0,\Z_{p_i}) = \frac{2}{p_i}\sum_{k=0}^{\frac{p_i-1}{2}}
  \sigma_{J^j_0}(e^{2\pi k \sqrt{-1}/p_i}) = 0
\]
since
$\frac{2\pi k}{p_i} \le \pi-\frac{\pi}{p_i} \le
\pi-\frac{\pi}{p_{j-1}} < d_j$ for $k=0,\ldots, \frac{p_i-1}{2}$.  It
follows that \ref{item:negativity-of-J_0} and
\ref{item:independence-of-signature-of-J_0} are satisfied.  Finally,
replace each $J^i_0$ with the connected sum of $N_i$ copies of $J^i_0$
for some $N_i > \frac{p_i}{4} \cdot 69\,713\,280\cdot(6n+90)$.  Then,
from \eqref{equation:nontriviality-of-signature-average}, the property
\ref{item:signature-of-J_0-large-enough} follows too.

\begin{remark}
  The above argument works for any given increasing sequence $\{p_i\}$
  of odd positive integers, without assuming that $p_i$ is prime.
\end{remark}

\section{Case 2: Use of Heegaard-Floer $d$-invariants}
\label{section:using-d-invariant}

Recall that Case 2 assumes that there is an $n$-negaton $X^-$ bounded
by $M(K_1)$ for which the kernel $P$ of $H_1(M(K_1);\Q[t^{\pm1}]) \to
H_1(X^-;\Q[t^{\pm1}])$ is equal to the subgroup~$\langle
\alpha_J\rangle$.  We will reach a contradiction under a weaker
hypothesis that there is a $1$-negaton $X^-$ satisfying $P=\langle
\alpha_J\rangle$.

Recall $K_1 = R(J^1_{n-1},D)$.  First, we claim that $J^1_{n-1}$ may
be replaced by the trivial knot, that is, we may assume that the knot
$K_0:=R(U,D)$ admits a $1$-negaton with the same kernel property.
This is due to \cite[Lemma~8.2]{Cochran-Harvey-Horn:2012-1}, which
essentially gives the following general statement:

\begin{lemma}[{\cite[Lemma~8.2]{Cochran-Harvey-Horn:2012-1}}]
  \label{lemma:eliminate-J-in-d-invariant-computation}  
  Suppose $K_1=K_0(\alpha,J)$ where $(K_0,\alpha)$ is a winding number
  zero pattern.  Note that the Alexander modules
  $H_1(M(K_0);\Q[t^{\pm1}])$ and $H_1(M(K_1);\Q[t^{\pm1}])$ are
  isomorphic, via the standard degree one map associated to a satellite construction.  If $J$ is unknotted by changing only positive
  crossings, then for every $1$-negaton $X_1$ bounded by $M(K_1)$,
  there is an $1$-negaton $X_0$ bounded by $M(K_0)$ such that the two
  maps $H_1(M(K_i);\Q[t^{\pm1}]) \to H_1(X_i;\Q[t^{\pm1}])$, $i=0,1$,
  have the same kernel under the identification of the Alexander
  modules.
\end{lemma}

In our case, $K_1=R(J^1_{n-1},D) = K_0(\alpha_J, J^1_{n-1})$ and
$J^1_{n-1}$ is unknotted by changing one positive crossing (the
topmost crossing in Figure~\ref{figure:stevedore-pattern}).  Thus the
claim follows by applying
Lemma~\ref{lemma:eliminate-J-in-d-invariant-computation}.  Note that
we use the assumption $n\geq 2$.

In what follows, we assume that $X^-$ is a 1-negaton bounded by
$M(K_0)$ such that $P:=\Ker\{H_1(M(K_0);\Q[t^{\pm1}]) \to
H_1(X^-;\Q[t^{\pm1}])\}$ is equal to $\langle\alpha_J\rangle$, where
$K_0=R(U,D)$ as above.

\subsection{Metabolizer of finite cyclic branched covers and
$d$-invariants}
\label{subsection:metabolizer-finite-covers}

For a positive integer $m$, let $\Sigma_m$ be the $m$-fold cyclic
cover of $S^3$ branched along~$K_0$.  Let $X_m$ be the 4-manifold
obtained by attaching a 2-handle, to the $m$-fold cyclic cover of 
the given $X^-$, along the pre-image of the zero framed meridian
of~$K_0$.  We have $\partial X_m=\Sigma_m$.  In case of $m=1$,
$\partial X_1=S^3$ and the cocore of the 2-handle is a slice disk in
$X_1$ bounded by~$K_0$.  (Indeed, our $X_1$ is the 4-manifold used to
give an alternative definition of negative knots in
\cite[Definition~2.2]{Cochran-Harvey-Horn:2012-1}.)  The manifold
$X_m$ is the $m$-fold branched cyclic cover of $X_1$ along the slice
disk.

To state the $d$-invariant obstruction
of~\cite{Cochran-Harvey-Horn:2012-1}, we use the following notations.
For a rational homology 3-sphere $Y$ and a \spinc\ structure
$\mathfrak{t}$ on $Y$, let $d(Y,\mathfrak{t})$ be the associated
correction term invariant of Ozsv\'ath and
Szab\'o~\cite{Ozsvath-Szabo:2003-2}.  When $Y$ is a $\Z_2$-homology
sphere (for instance it is the case for $Y=\Sigma_m$ if $m$ is an odd
prime power), denote by $\mathfrak{s}_Y$ the \spinc\ structure induced
by the unique spin structure on~$Y$.  Every \spinc\ structure on $Y$
is of the form $\mathfrak{s}_Y+c$ for some $c\in H^2(Y)$, where $+$
designates the action of $H^2(Y)$ on the \spinc\ structures.  A
subgroup $G$ in $H_1(Y)$ is called a \emph{metabolizer} if
$G = G^\perp := \{x\in H_1(Y)\mid\lambda(x,G)=0\}$, where
$\lambda\colon H_1(Y)\times H_1(Y)\to \Q/\Z$ is the linking form.

\begin{definition-named}[Correction term $d$-invariant obstruction
  \textnormal{(\cite[Theorem~6.5]{Cochran-Harvey-Horn:2012-1})}]

  If $X^-$ is a $1$-negaton bounded by $M(K_0)$ and $m$ is an odd
  prime power, then $G:= \Ker\{H_1(\Sigma_m)\rightarrow H_1(X_m)\}$ is
  a metabolizer, and
  \begin{equation}
    \label{equation:CHHinequality} 
    d(\Sigma_m,\mathfrak{s}_{\Sigma_m}+\hat{x})\geq 0 
  \end{equation}
  for the Poincar\'e dual $\hat{x}\in H^2(\Sigma_m)$ of every
  $x\in G$.
\end{definition-named}

Understanding the metabolizer in the $d$-invariant obstruction is
essential for our purpose.  We will relate the \emph{above}
metabolizer $G$ of the $m$-fold branched cover to the metabolizer
$P\subset H_1(M(K_0);\Q[t^{\pm1}])$ of the infinite cyclic cover.
Recall that $H_1(\Sigma_m)$ is isomorphic to
$H_1(M(K_0);\Z[t^{\pm1}])/\langle t^m-1 \rangle$ (for instance
see~\cite{Milnor:1968-1}, or use a Wang sequence argument).  Let
$x_1$, $x_2 \in H_1(\Sigma_m)$ be the images of
$[\alpha_J], [\alpha_D] \in H_1(M(K_0);\Z[^{\pm1}])$ respectively.
First, we claim that a metabolizer in $H_1(\Sigma_m)$ is either
$\langle x_1\rangle$ or~$\langle x_2\rangle$.  In fact, from the
previous computation of $H_1(M(K_0);\Z[t^{\pm1}])$ and the Blanchfield
form in
Section~\ref{section:construction-of-examples-first-step-proof}, it
follows that each of $x_1$, $x_2$ generates a cyclic subgroup
$\langle x_i\rangle \subset H_1(\Sigma_m)$ of order $2^m-1$,
$H_1(\Sigma_m) = \langle x_1\rangle \oplus \langle x_2\rangle$, and
the linking form satisfies $\lambda(x_1,x_1) = \lambda(x_2,x_2)=0$ and
$\lambda(x_1,x_2)\ne 0$.  (A computation confirming this can also be
found in~\cite[Proposition~2]{Gilmer-Livingston:1992-1}.)  The claim
follows from this.

\begin{lemma}
  \label{lemma:metabolizer-finite-cover}
  Under the hypothesis of Case
  2 that $P=\langle\alpha_J\rangle$, $G=\Ker\{H_1(\Sigma_m)\to
  H_1(X_m)\}$ is equal to $\langle x_1\rangle$ for every sufficiently
  large prime~$m$.
\end{lemma}

We do not know any estimate for how large $m$ should be.  In our
argument below, it depends on the 1-negaton~$X^-$.  This is the reason
that we need to consider an infinite family of the $d$-invariants.

\begin{proof}[Proof of Lemma~\ref{lemma:metabolizer-finite-cover}]
  Consider the following commutative diagram:
  \[
    \begin{tikzcd}
      H_1(M(K_0);\Q[t^{\pm 1}]) \ar[d] &
      H_1(M(K_0);\Z[t^{\pm 1}]) \ar[l]\ar[r]\ar[d] &
      H_1(\Sigma_m) \ar[d]
      \\
      H_1(X^-;\Q[t^{\pm 1}]) &
      H_1(X^-;\Z[t^{\pm 1}]) \ar[l]\ar[r] &
      H_1(X_m)
    \end{tikzcd}
  \]
  Here, the vertical maps are induced by inclusions, the left
  horizontal maps are tensoring by $\Q$, and the right horizontal maps
  are the quotient maps.  The left and right vertical maps have
  kernels $P$ and $G$ respectively.
  
  The image of $[\alpha_J]\in H_1(M(K_0);\Z[t^{\pm 1}])$ in
  $H_1(X^-;\Z[t^{\pm 1}])$ has finite order, say~$a$, since
  $[\alpha_J]$ is sent to $0\in H_1(X^-;\Q[t^{\pm1}])$.  Choose an odd
  prime $m$ not smaller than each prime factor of~$a$.  We need the
  following elementary fact.
  
  \begin{lemma}
    \label{lemma:prime}
    If $p$ and $m$ are primes and $p\le m$, then $p$ and $2^m-1$ are
    coprime.
  \end{lemma}

  \begin{proof}
    If $p=2$, the conclusion is straightforward.  Suppose $p$ is odd
    and $p\mid 2^m-1$.  Let $d$ be the multiplicative order of $2$ in
    $\Z_p^\times$.  Since $2^m\equiv 1 \pmod p$, $d$ divides $m$.
    Since $m$ is a prime, $d=m$.  Since $2^{p-1}\equiv 1\pmod p$ by
    Fermat's little theorem, it follows that $p-1$ is a multiple of
    $m$.  This contradicts $p\le m$. 
  \end{proof}

  Returning to the proof of
  Lemma~\ref{lemma:metabolizer-finite-cover}, it follows that every
  prime factor of $a$ is coprime to $2^m-1$, by Lemma~\ref{lemma:prime}.
  So $a$ is coprime to $2^m-1$.  Choose $b$ such that $ab\equiv 1\pmod
  {2^m-1}$.  Then the image $ab\cdot[x_1]\in H_1(\Sigma_m)$ of
  $ab\cdot [\alpha_J] \in H_1(M(K_0);\Z[t^{\pm1}])$ is equal to
  $[x_1]$, since $[x_1]$ has order $2^m-1$.  Also, by the choice
  of~$a$, the image of $ab\cdot[\alpha_J]$ in $H_1(X^-;\Z[t^{\pm 1}])$
  is zero.  It follows that $[x_1]$ lies in $G$, and thus $G=\langle
  x_1\rangle$.  This completes the proof of
  Lemma~\ref{lemma:metabolizer-finite-cover}.
\end{proof}

Now, our argument for Case 2 proceeds as follows.  For an odd prime
$m$, recall that $\mathfrak{s}_{\Sigma_m}$ is the \spinc\ structure of the $\Z_2$-homology 3-sphere $\Sigma_m$ induced by the unique spin structure.
Theorem~\ref{theorem:d-invariant}, which is stated below, implies that
$d(\Sigma_m, \mathfrak{s}_{\Sigma_m}+2^{m-1}\hat x_1)$ is negative for
every odd prime~$m$.  Since $G=\langle x_1 \rangle$ for sufficiently
large $m$ by Lemma~\ref{lemma:metabolizer-finite-cover}, this
contradicts the $d$-invariant
obstruction~\eqref{equation:CHHinequality}.  This completes the proof
for Case 2, modulo the proof of Theorem~\ref{theorem:d-invariant}.

\subsection{Computation of the $d$-invariants}

The remaining part of this paper is devoted to prove the following:

\begin{theorem}
  \label{theorem:d-invariant}
  For every odd prime power $m$,
  $d(\Sigma_m, \mathfrak{s}_{\Sigma_m}+2^{m-1}\widehat{x}_1)\leq
  -\frac{3}{2}$.
\end{theorem}

We remark that the new contribution is the case of $m>3$; for $m=3$,
Theorem~\ref{theorem:d-invariant} was shown in the earlier work of
Cochran, Harvey and Horn~\cite{Cochran-Harvey-Horn:2012-1}.  In
addition, part of our proof which is given in
Section~\ref{subsection:description-of-W} essentially follows the
arguments in~\cite{Cochran-Harvey-Horn:2012-1}.  Our arguments in
Section~\ref{subsection:proof-of-properties-of-W}, which prove a key
lemma (see Lemma~\ref{lemma:the-cobordism-W}) for general $m>3$, are
new and use a different approach.

\begin{remark}\label{remark-spinc}
  In the proof of Theorem~\ref{theorem:d-invariant}, we will use various standard facts on \spinc\ structures, their Chern class, and cohomology classes.
  We briefly recall them below, for the reader's convenience, and to fix notation. 
  Let $\Spinc(M)$ be the set of \spinc\ structures on a manifold~$M$. 
  As in Section~\ref{subsection:metabolizer-finite-covers}, denote the action of $x\in H^2(M)$ on $\Spinc(M)$ by $\mathfrak t \mapsto \mathfrak t + x$ for $\mathfrak{t}\in \Spinc(M)$.
  When $\Spinc(M)$ is nonempty, $x\mapsto \mathfrak{t}+x$ is a bijection $H^2(M) \xrightarrow[\approx]{} \Spinc(M)$ for each $\mathfrak{t}\in \Spinc(M)$.


  \begin{enumerate}
  \item\label{item-remark-spinc-1}
  For $\mathfrak{t}\in\Spinc(M)$, denote by $c_1(\mathfrak t)\in H^2(M)$ the first Chern class of the associated determinant line bundle.
  It satisfies $c_1(\mathfrak{t}+x) = c_1(\mathfrak{t})+2x$.
  \item\label{item-remark-spinc-2} 
  For a $\Z_2$-homology 3-sphere $Y$, $c_1\colon \Spinc(Y)\to H^2(Y)$ is bijective since $H^2(Y)$ does not have 2-torsion.
  For the \spinc\ structure $\mathfrak s_Y$ induced by the unique spin structure of $Y$, $c_1(\mathfrak s_Y)=0$.
  So, $c_1(\mathfrak s)=0$ if and only if $\mathfrak s = \mathfrak s_Y$.
  \item\label{item-remark-spinc-3}  Let $X$ be a 4-manifold. Then the image of $c_1\colon \operatorname{Spin}^c(X)\to H^2(X)$ is equal to the set of characteristic classes.
  Here a cohomology class $c\in H^2(X)$ is \emph{characteristic} if $c(x)\equiv x\cdot x\pmod 2$ for all $x\in H_2(X)$, where $c(x)$ is the evaluation of $c$ on the homology class $x$ and~$\cdot$~is the intersection pairing.
  This definition is equivalent to that $c(x)\equiv x\cdot x\pmod{2}$ holds for generators $x$ of $H_2(X)$.
  \item\label{item-remark-spinc-4} In addition, suppose that each component of $\partial X$ is a rational homology 3-spheres.
  Then for $c\in H^2(X)$, $c^2\in \Q$ is defined as follows.
  Let $j\colon H_2(X)\to H_2(X,\partial X)$ be inclusion-induced.
  Since $H_1(\partial X)$ is torsion, $nc$ is the Poincar\'{e} dual of $j(x)$ for some nonzero $n\in\Z$ and $x\in H_2(X)$.
  We define $c^2=(x\cdot x)/n^2$.
  \end{enumerate}
\end{remark}

To prove Theorem~\ref{theorem:d-invariant}, we will use a cobordism
given in the following lemma.  Let
\[
  Y_m =\big((m-3)L(3,1)\big) \# S^3_{-7}(D\#-T)\#
  S^3_{-6}(-T),
\]
where $L(3,1)$ is the lens space, $-T$ is the left handed
trefoil knot, and $S^3_r(K)$ designates the $r$-framed surgery
manifold of~$K\subset S^3$.

\begin{lemma}
  \label{lemma:the-cobordism-W}
  For every odd prime power $m$, there exists a cobordism $W$ bounded by
  $(-\Sigma_m)\sqcup Y_m$ and a \spinc\ structure $\mathfrak{t}$ on
  $W$ satisfying the following\textup{:}
  \begin{enumerate}[label=({W\arabic*})]
  \item\label{item:negative-definiteness-of-W} $W$ is negative
    definite and $\beta_2(W)=3m-3$.
  \item\label{item:spinc-structure-on-W}
    $c_1(\mathfrak{t}|_{\d W})=(\hat{x}_{1},0) \in H^2(\Sigma_m)
    \oplus H^2(Y_m) = H^2(\partial W) $ and $c_1(\mathfrak{t})^2=-m$.
  \end{enumerate}
\end{lemma}

\begin{proof}[Proof of Theorem~\ref{theorem:d-invariant}]
  Let $W$ and $\mathfrak{t}$ be those given by Lemma~\ref{lemma:the-cobordism-W}.
  We will first prove that $\mathfrak{t}|_{\Sigma_m}=\mathfrak{s}_{\Sigma_m}+2^{m-1}\hat{x}_1$
  and will obtain the desired conclusion by applying the $d$-invariant inequality to $(W,\mathfrak{t})$.
  We have $c_1(\mathfrak{t}|_{\Sigma_m})=\hat{x}_1$ and
  $c_1(\mathfrak{t}|_{Y_m})=0$ by~\ref{item:spinc-structure-on-W}.
  By 
  Remark~\ref{remark-spinc}~\eqref{item-remark-spinc-1} and~\eqref{item-remark-spinc-2}, and since $x_1$ and therefore $\hat{x}_1$ has order $2^m-1$, we have
  \[
    c_1(\mathfrak{s}_{\Sigma_m}+2^{m-1}\hat x_1) = 2^m\cdot \hat x_1 =
    \hat x_1 = c_1(\mathfrak{t}|_{\Sigma_m}).
  \]
  It follows that
  $\mathfrak{t}|_{\Sigma_m}=\mathfrak{s}_{\Sigma_m}+2^{m-1}\hat{x}_1$, since
  $c_1$ is 1-1 as mentioned in Remark~\ref{remark-spinc}~\eqref{item-remark-spinc-2}.

  Since $W$ is negative definite, the $d$-invariant inequality of
  Ozsv\'ath-Szab\'o~\cite[Theorem~9.6]{Ozsvath-Szabo:2003-2} gives
  \begin{equation*}
    \label{equation:inequality}
    d(Y_m,\mathfrak{t}|_{Y_m}) -
    d(\Sigma_m,\mathfrak{s}_{\Sigma_m}+2^{m-1}\hat{x}_{1})
    \ge \frac{c_1(\mathfrak{t})^2+\beta_2(W)}{4}=\frac{2m-3}{4},
  \end{equation*}
  where the equality follows from
  Lemma~\ref{lemma:the-cobordism-W}~\ref{item:negative-definiteness-of-W}
  and~\ref{item:spinc-structure-on-W}.  
  
  The value of $d(Y_m,\mathfrak{t}|_{Y_m})$ is computed by using known
  results, as described below.  For brevity, denote
  $A:=S^3_{-7}(D\#-T)$ and $B:=S^3_{-6}(-T)$ temporarily, so that
  $Y_m=\big((m-3)L(3,1)\big)\#A\#B$. Since
  $c_1(\mathfrak{t}|_{Y_m})=0$, $\mathfrak{t}$ restricts to a \spinc\
  structure whose first Chern class is trivial on each summand.
  By Remark~\ref{remark-spinc}~\eqref{item-remark-spinc-2}, $\mathfrak{t}|_{L(3,1)} = \mathfrak{s}_{L(3,1)}$ and
  $\mathfrak{t}|_{A} = \mathfrak{s}_{A}$ since $L(3,1)$ and $A$ are $\Z_2$-homology spheres.
  By the recursive $d$-invariant formula for lens spaces given in
  \cite[Proposition~4.8]{Ozsvath-Szabo:2003-2},
  $d(L(3,1),\mathfrak{s}_{L(3,1)})=\frac12$.  Due to Cochran, Harvey
  and Horn~\cite[p.~2151]{Cochran-Harvey-Horn:2012-1},
  $d(A,\mathfrak{s}_{A})= -\frac32$.  Using the formula of Owens and
  Strle for $L$-space knots~\cite[Theorem~6.1]{Owens-Strle:2012-1}, we
  have $d(B,\mathfrak{t}|_B)\le \frac 34$.  It follows that
  \[
    d(Y_m,\mathfrak{t}|_{Y_m}) = (m-3)\cdot
    d(L(3,1),\mathfrak{s}_{L(3,1)}) + d(A,\mathfrak{s}_{A}) +
    d(B,\mathfrak{t}|_{B}) \leq \frac{2m-9}{4}.
  \]
  Combining the above, we obtain
  \[
    d(\Sigma_m,\mathfrak{s}_{\Sigma_m}+2^{m-1}\hat{x}_{1})\leq \frac{2m-9}{4} -
    \frac{2m-3}{4} = -\frac{3}{2}.
  \]
  This completes the proof of Theorem \ref{theorem:d-invariant} modulo
  the proof of Lemma~\ref{lemma:the-cobordism-W}.
\end{proof}

\section{The cobordism $W$ from $\Sigma_m$ to $Y_m$}
\label{section:the-cobordism-W}

In this section, we describe the construction of $W$ and perform some
computation on $W$ to prove Lemma~\ref{lemma:the-cobordism-W}.

\subsection{Construction of $W$}
\label{subsection:description-of-W}

\begin{figure}[t]
  \includestandalone{covering-surgery-diagram}
  \caption{A surgery presentation of $\Sigma_m$ consisting of 0-framed
    curves, drawn with additional curves~$v_i$.}
  \label{figure:covering-surgery-diagram}

  \vskip\floatsep
  
  \includestandalone{unknotting-surgery-curve}
  \caption{The curve $v_i$ ($i=3m-4,\ldots,4m-6$) inside the $D$ box,
    where $-3$ and $-2$ designate negative full twists of the involved
    strands.}
  \label{figure:unknotting-surgery-curve}
\end{figure}

Consider the diagram in Figure~\ref{figure:covering-surgery-diagram},
which consists of $(2m-2)$ 0-framed curves and additional curves
$v_1,\ldots,v_{3m-5}$.
The labels $v_{3m-4},\ldots,v_{4m-6}$ will be
explained in the next paragraph.  For now, ignore the arrows and
labels~$x_i$; they will be used in
Section~\ref{subsection:proof-of-properties-of-W}.  The 0-framed
curves form a standard surgery diagram of the $m$-fold cyclic branched
cover $\Sigma_m$ of $K_0$, which is obtained from
Figure~\ref{figure:base-seed-knot} using the Akbulut-Kirby
method~\cite{Akbulut-Kirby:1979-1}.  For later use, note that
$v_{3m-5}$ is (isotopic to) a lift of the curve $\alpha_J$ in
Figure~\ref{figure:base-seed-knot}.  So, choosing appropriate
basepoints, we may assume that the surjection
$H_1(M(K_0);\Z[t^{\pm1}]) \to H_1(\Sigma_m)$ takes $[\alpha_J]$
to~$[v_{3m-5}]$, that is, $[v_{3m-5}]$ is equal to
$x_1\in H_1(\Sigma_m)$ used in
Section~\ref{section:using-d-invariant}.

Regard $v_1,\ldots,v_{3m-5}$ as curves in~$\Sigma_m$.  The following
observation will be useful: $(-1)$-surgery along $v_1,\ldots,v_{3m-6}$
changes the enclosed crossings (at the cost of framing changes of the
$0$-framed components), and after the crossing changes, we would be
able to isotope the resulting $2m-2$ curves into $m-1$ split
2-component links, if $D$ were trivial.  Furthermore, using the fact
that the Whitehead double $D$ is unknotted by changing a positive
crossing, we could also do additional $(-1)$ surgeries to remove the
$D$ boxes.  More precisely, for each of the $D$ boxes except the
leftmost one, let $v_i$ ($i=3m-4,\ldots,4m-6$) be the curve inside the
box shown in Figure~\ref{figure:unknotting-surgery-curve}.  (For the
rightmost $D$ box, replace the double strands in
Figure~\ref{figure:unknotting-surgery-curve} with a single strand.)
Enumerate them from right to left, as indicated by the labels
$v_{3m-4},\ldots,v_{4m-6}$ in
Figure~\ref{figure:covering-surgery-diagram}.  The $(-1)$-surgery
along $v_i$ ($i=3m-4,\ldots,4m-6$) would eliminate the enclosing $D$
box.

Now, take $\Sigma_m\times[0,1]$, and attach $(4m-6)$ 2-handles to
$\Sigma_m\times 1$ along the $(-1)$-framing of the curves
$v_1,\ldots,v_{4m-6}$, to obtain a 4-manifold which we temporarily
call~$N$.  The promised 4-manifold $W$ will be obtained by attaching
$m-3$ additional 3-handles to~$N$.  The attaching 2-spheres are described as
follows.  Regard Figure~\ref{figure:covering-surgery-diagram}, with
each $v_i$ $(-1)$-framed, as a surgery diagram of the upper boundary
of $N$ ($=\partial N\sm \Sigma_m$).  Use the effect of the
$(-1)$-surgery discussed above, and perform isotopy, to obtain a
simplified surgery diagram of the upper boundary shown in
Figure~\ref{figure:after-m1-surgery}.  For each of the middle $m-3$
sublinks with two $3$-framed components, perform handle slide of one
of the components over the other.  This gives the surgery diagram in
Figure~\ref{figure:after-handle-slides}.  The $0$-framed unknotted
circles in this diagram give $(m-3)$ $S^1\times S^2$ summands of the
upper boundary of~$N$.  Attach, to $N$, $(m-3)$ 3-handles along the
$S^2$ factors of these summands.  The result is our 4-manifold~$W$.

\begin{figure}[H]
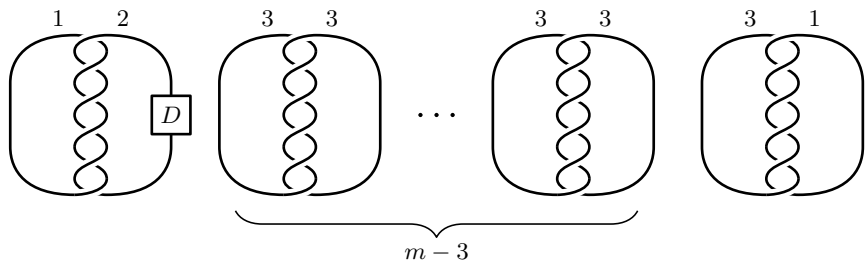

  \includestandalone{after-m1-surgery}
  \caption{The result of $(-1)$-surgery (or blowing down) along
    the~$v_i$.}
  \label{figure:after-m1-surgery}
\end{figure}

\begin{figure}[H]
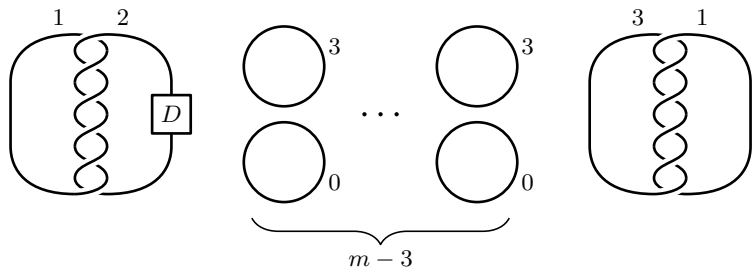

  \includestandalone{after-handle-slides}  
  \caption{The result of handle slides.}
  \label{figure:after-handle-slides}
\end{figure}

We need to verify that the upper boundary of~$W$ is~$Y_m$.  Since the
3-handle attachments eliminate the $(m-3)$ 0-framed unknotted
circles, Figure~\ref{figure:after-handle-slides} with the 0-framed
circles removed is a surgery description of the upper boundary of~$W$.
By the two $(+1)$-surgeries in the diagram, it becomes the connected
sum of $(-7)$-surgery on $D\#-T$, $(-6)$-surgery on $-T$, and $m-3$
copies of~$L(3,1)$.  This is the 3-manifold $Y_m$, as desired.

\subsection{Definiteness and \spinc\ structure computation for $W$}
\label{subsection:proof-of-properties-of-W}

This subsection is devoted to the proof of
Lemma~\ref{lemma:the-cobordism-W}.  Recall that
Lemma~\ref{lemma:the-cobordism-W}~\ref{item:negative-definiteness-of-W}
asserts that $W$ is negative definite and $b_2(W)=3m-3$.  Since an
explicit handle decomposition of $W$ is given, one may try to compute
directly the second homology and intersection form to prove the
assertions.  Though, after several attempts, we learned that proving
the desired definiteness \emph{for all $m$} by direct computation was
difficult, if feasible, because of the growth of the size and
sophistication of the intersection matrix.

In what follows, we present a different approach.  The idea is to
obtain an ``easier'' 4-manifold (which we call $\tilde W$ below) by
attaching another ``easier'' 4-manifold (which we call $W_0$ below) to
$W$, motivated from the above surgery diagram calculus of 3-manifolds.
Then we investigate $W$ as the difference of the two easier ones.

\begin{proof}[Proof of
  Lemma~\ref{lemma:the-cobordism-W}~\ref{item:negative-definiteness-of-W}]

  Let $W_0$ be the 4-manifold with one 0-handle and $(2m-2)$ 2-handles
  attached along the 0-framed curves shown in
  Figure~\ref{figure:covering-surgery-diagram}.  That is,
  Figure~\ref{figure:covering-surgery-diagram} with the $v_i$
  forgotten is now viewed as a Kirby diagram of~$W_0$.  It is
  straightforward that $\partial W_0=\Sigma_m$.  Let
  $\tilde W=W_0 \cupover{\Sigma_m} W$. See the schematic diagram in Figure~\ref{figure:W-tilde} (for now, ignore the thick arcs and labels $a_i^{\vphantom{1}}\sigma_i'$, $a_j^{\vphantom{1}}\sigma_j'$, $z_i'$, $z_j'$ and $a_j^{\vphantom{1}} v_j'$, which will be used later).

\begin{figure}[H]
  \includestandalone{W-tilde}
  \caption{
  The manifolds $W$ and $\tilde{W}=W_0\cupover{\Sigma_m}W$.
  Recall that $W$ is obtained by attaching, to $\Sigma_m\times [0,1]$, $(4m-6)$ 2-handles and $(m-3)$ 3-handles where the 2-handles are attached along the curves $v_i\times 1 \subset \Sigma_m\times 1$.  
  Thick arcs schematically represent 2-cycles $E_i'=a_i^{\protect\vphantom1}\sigma_i'-z_i'$ and $E_j'=a_j^{\protect\vphantom1}\sigma_j'-z_j'$, which will be described and used in the proof of Lemma~\ref{lemma:dual-of-w-is-characteristic}.
  }
  \label{figure:W-tilde}
\end{figure}

By a Mayer-Vietoris argument, using that $\Sigma_m$ is a rational
  homology sphere, we obtain
  \begin{equation}
    \label{equation:additivity-b2}
    b_2(\tilde W) = b_2(W) + b_2(W_0).
  \end{equation}
  Also, by Novikov additivity, the signatures are related
  as follows:
  \begin{equation}
    \label{equation:additivity-signature}
    \sign(\tilde W) = \sign(W) + \sign(W_0).
  \end{equation}

  First, we compute $b_2(W_0)$ and $\sign(W_0)$.  Since $W_0$ consists
  of $(2m-2)$ 2-handles and one 0-handle, $b_2(W_0)=2m-2$.  By
  Akbulut-Kirby~\cite{Akbulut-Kirby:1979-1}, $W_0$ is the $m$-fold
  cyclic cover of $D^4$ branched along the surface obtained by pushing
  into $D^4$\ the interior of the obvious Seifert surface of genus one
  for $K_0=R(U,D)$.  Therefore, by a well known fact (see, for
  instance, \cite{Viro:1973-1} and \cite[Section~12]{Gordon:1978-1}),
  $\sign(W_0)$ is determined by the Levine-Tristram signature function
  $\sigma_{K_0}$, as follows:
  \[
    \sign(W_0)=\sum_{k=0}^{m-1}\sigma_{K_0}(e^{2\pi k\sqrt{-1}/m}).
  \]
  Since $K_0$ is slice, $\sigma_{K_0}(\omega)=0$ when $\omega$ is a
  root of unity of prime power order.  It follows that
  $\sign(W_0)=0$.

  To compute $\beta_2(\tilde W)$ and $\sign(\tilde W)$, we use an
  alternative description of $\tilde W=W_0\cupover{\Sigma_m}W$.  By
  definition, $\tilde W$ consists of one 0-handle, $(2m-2)$ 2-handles
  attached along the 0-framed curves in
  Figure~\ref{figure:covering-surgery-diagram}, $(4m-6)$ 2-handles
  attached along the $(-1)$-framing of the curves
  $v_1,\ldots,v_{4m-6}$ in
  Figure~\ref{figure:covering-surgery-diagram}, and additional $(m-3)$
  3-handles.  Blow down $\tilde W$ $(4m-6)$ times, to realize the effect
  of the $(-1)$ surgery.  This transforms
  Figure~\ref{figure:covering-surgery-diagram} (with each $v_i$
  $(-1)$-framed) to Figure~\ref{figure:after-m1-surgery}.  That is, we
  have
  \begin{equation}
    \label{equation:blowdown-on-W}
    \tilde W = W' \# \big((4m-6)\overline{\C P^2}\big)
  \end{equation}
  where the result $W'$ of blowing down is given by
  Figure~\ref{figure:after-m1-surgery}, viewed as a 4-manifold Kirby
  diagram now, with additional $(m-3)$ 3-handles.  By handle slides
  and cancellations of 3-handles with 2-handles, it follows that
  Figure~\ref{figure:after-handle-slides} with the $(m-3)$ 0-framed
  circles removed is a Kirby diagram for~$W'$ (without 3-handles).

  Now, since our final Kirby diagram consists of $(m+1)$ 2-handles
  without any 1- and 3-handles, $b_2(W')=m+1$.  Also, the intersection
  matrix for $W'$ is the direct sum of $\sbmatrix{1 & 3\\ 3 & 2}$,
  $\sbmatrix{3 & 3\\ 3 & 1}$ and $m-3$ copies of the $1\times 1$
  matrix~$[\,3\,]$.  Since the two $2\times 2$ submatrices have
  vanishing signatures, $\sign(W')=m-3$.  From this and
  \eqref{equation:blowdown-on-W}, it follows that $b_2(\tilde W)=5m-5$
  and $\sign(\tilde W)=-3m+3$.

  By substituting the above values into \eqref{equation:additivity-b2}
  and \eqref{equation:additivity-signature}, it follows that
  $b_2(W)=3m-3$ and $\sign(W)=-3m+3$.  This completes the proof of
  Lemma~\ref{lemma:the-cobordism-W}~\ref{item:negative-definiteness-of-W}.
\end{proof}

In the proof of
Lemma~\ref{lemma:the-cobordism-W}~\ref{item:spinc-structure-on-W}, the
following lemma is essential.  Recall that $W$ has $(4m-6)$ 2-handles
attached to $\Sigma_m \times [0,1]$ along the $(-1)$-framed curves
$v_i\times 1$ (see Figures~\ref{figure:covering-surgery-diagram} and 
\ref{figure:W-tilde}).  For
$1\le i \le 4m-6$, let $\sigma_i$ be the union of $v_i\times [0,1] \subset \Sigma_m\times[0,1]$
and the core of the 2-handle attached to $v_i\times 1$.
Orient $v_i$ in
Figure~\ref{figure:covering-surgery-diagram} counterclockwise, and
orient the 2-disk $\sigma_i$ in such a way that $\partial\sigma_i=v_i\times 0$. That is, homologically, the image of the relative class $[\sigma_i]$ under the boundary map
\[
  \partial\colon H_2(W,\partial W)\to H_1(\partial W)=H_1(\Sigma_m)
  \oplus H_1(Y_m)
\]
is $([v_i],0)$.  Let
$w=[\sigma_{3m-5}] +\cdots+ [\sigma_{4m-6}] \in H_2(W,\d W)$.

\begin{lemma}
  \label{lemma:dual-of-w-is-characteristic}
  The Poincar\'e dual $\hat w\in H^2(W)$ is characteristic and
  satisfies $\hat w ^2 = -m$.
\end{lemma}

\begin{proof}[Proof of
  Lemma~\ref{lemma:the-cobordism-W}~\ref{item:spinc-structure-on-W}]
  
  Since $\hat w$ in Lemma~\ref{lemma:dual-of-w-is-characteristic} is
  characteristic, there is a \spinc\ structure $\mathfrak{t}$ on $W$
  such that $c_1(\mathfrak{t})=\hat{w}$ by Remark~\ref{remark-spinc}~\eqref{item-remark-spinc-3}.  We will verify that
  $\mathfrak{t}$ satisfies the desired properties
  $c_1(\mathfrak{t})^2=-m$ and
  $c_1(\mathfrak{t}|_{\d W}) = (\hat x_1,0)\in H^2(\partial
  W)=H^2(\Sigma_m)\oplus H^2(Y_m)$.

  First, $c_1(\mathfrak{t})^2=\hat{w}^2=-m$.  By naturality, we have
  \[
    c_1(\mathfrak{t}|_{\d W}) = c_1(\mathfrak{t})|_{\d W} = \hat
    w|_{\partial W} = \widehat{\partial w}
  \]
  where $\widehat{\partial w}\in H^2(\partial W)$ is the Poincar\'e
  dual of $\d w\in H_1(\d W)$.  By the first paragraph of
  Section~\ref{subsection:description-of-W}, $[v_{3m-5}]=x_1$ and thus
  $\d [\sigma_{3m-5}] = (x_1,0)$.  For $i\ge 3m-4$, we have
  $\d [\sigma_{i}] =([v_i],0)=0$, since $v_i$ has linking number zero
  with other surgery curves in
  Figure~\ref{figure:unknotting-surgery-curve}.  It follows that
  $\d w = (x_1,0)$.  Therefore
  $c_1(\mathfrak{t}|_{\d W})=(\hat{x}_{1},0)$.
\end{proof}

The remaining part of this section is devoted to proving
Lemma~\ref{lemma:dual-of-w-is-characteristic}, which is largely
computations of intersection data in terms of rational linking
numbers.

\begin{proof}[Proof of Lemma~\ref{lemma:dual-of-w-is-characteristic}]
  
  Let $w_0=(2^m-1)w\in H_2(W,\d W)$.
  Since $2^m-1$ is odd, $\hat w$ is characteristic if and only if so is $\hat w_0$, by the mod 2 defining property of a characteristic class in Remark~\ref{remark-spinc}~\eqref{item-remark-spinc-3}.
  Also, 
  $\hat{w}_0^2=(2^m-1)^2 \hat w^2$, by the definition of $(-)^2$ in Remark~\ref{remark-spinc}~\eqref{item-remark-spinc-4}.
  So, Lemma~\ref{lemma:dual-of-w-is-characteristic} is equivalent to that
  $\hat{w}_0$ is characteristic and
  $\hat{w}_0^2=-(2^m-1)^2m$.
  We will prove this.

  First, we construct a (non-relative) cycle representative $E_0$ of
  the class~$w_0$, that is, $j[E_0]=w_0$ where 
  $j\colon H_2(W)\to H_2(W, \d W)$ is the
  inclusion-induced map.  Let
  \[
    a_i=\begin{cases}
      2^m-1&\textrm{for }i=1,\ldots,3m-5,
      \\
      1&\textrm{for }i=3m-4,\ldots,4m-6.
    \end{cases}
  \]
  Then $a_iv_i$ is null-homologous in $\Sigma_m$ for all~$i$.  Indeed,
  for $i\le 3m-5$, it is straightforward since
  $H_1(\Sigma_m) \cong \Z_{2^m-1}\oplus\Z_{2^m-1}$, and for
  $i\ge 3m-4$, $v_i$ shown in Figure~\ref{figure:unknotting-surgery-curve} is null-homologous since $v_i$ has linking number
  zero with other surgery curves in
  Figure~\ref{figure:covering-surgery-diagram}.  Choose a 2-chain
  $z_i$ in $\Sigma_m\subset W$ whose oriented boundary is~$a_iv_i$. Let
   $E_i$ be the $2$-cycle $a_i\sigma_i-z_i$ in~$W$. 
   
  Recall that $W$ is obtained by attaching 2-handles to $\Sigma_m\times
  [0,1]$ along the $v_i$.  It can be seen that the classes $[E_i]$
  ($i=1,\ldots,4m-6$) generate a subgroup $\langle [E_i]\rangle$ with
  odd index in~$H_2(W)$. In fact, since $H_2(\Sigma_m)=H_1(W)=0$, a
  Mayer-Vietoris argument gives us
  \[
    0\to H_2(W)\to \Z^{4m-6} \to H_1(\Sigma_m) \to 0.
  \]
  So $H_2(W)$ is a subgroup of $\Z^{4m-6}$ with index equal to
  $|H_1(\Sigma_m)|=(2^m-1)^2$.  By the definition of $E_i$ and $a_i$, we
  have $[\Z^{4m-6}:\langle [E_i]\rangle]=(2^m-1)^{3m-5}$.  It follows
  that $[H_2(W):\langle [E_i] \rangle]=(2^m-1)^{3m-3}$, as claimed.
  
  Let
  \begin{equation}
    E_0=E_{3m-5} + (2^m-1)(E_{3m-4}+\cdots+E_{4m-6}).
    \label{equation:definition-of-E0}
  \end{equation}
  Since $j[E_i]=a_i[\sigma_i]$, we have
  $j[E_0]=(2^m-1)([\sigma_{3m-5}]+\cdots+[\sigma_{4m-6}])=w_0$. It
  follows that $\hat{w}_0^2=E_0\cdot E_0$ and $\hat{w}_0(E_i)=E_i\cdot
  E_0$, where $\cdot$ denotes the intersection in~$W$.  (As usual, the
  intersection of two chains is computed by taking a pushoff of one of
  them which is transverse to another.)  We will verify that 
  \begin{equation}
    \label{equation:spinc-a}
    \begin{aligned}
    E_0\cdot E_0 &=-(2^m-1)^2m, \\
    E_i\cdot E_0 &\equiv E_i\cdot E_i\pmod{2}
    \end{aligned}
  \end{equation}
  for $i=1,\ldots,4m-6$. Since $\langle [E_i]\rangle$ has odd index in
  $H_2(W)$, the second identity in \eqref{equation:spinc-a} implies that
  $\widehat w_0$ is characteristic. See Remark~\ref{remark-spinc}\eqref{item-remark-spinc-3}.  So the verification of
  \eqref{equation:spinc-a} completes the proof of
  Lemma~\ref{lemma:dual-of-w-is-characteristic}.

  Recall that for every pair of two disjoint 1-cycles $(\alpha,\beta)$
  in a rational homology 3-sphere~$\Sigma$, the linking number
  $\lk_{\Sigma}(\alpha, \beta)\in \Q$ is defined as follows: if $u$ is
  a 2-chain bounded by $r\alpha$ in $\Sigma$ for some nonzero integer
  $r$, then $\lk_{\Sigma}(\alpha, \beta) = \frac1r (u\circ \beta)$
  where $\circ$ denotes temporarily the intersection in~$\Sigma$.

  In our case, recall that $E_i = a_i\sigma_i -z_i$, where $\sigma_i$ is the union of $v_i \times[0,1] \subset \Sigma_m \times [0,1] \subset W$ and the core of the 2-handle attached along $v_i\times 1\subset \Sigma_m\times 1$, oriented such that $\partial\sigma_i=v_i\times 0$, and $z_i$ is a 2-chain in $\Sigma_m=\Sigma_m\times 0$ such that $\partial z_i = a_i v_i$.
  Fix $i$ and~$j$ ($1\le i$, $j\le 4m-6$). One can push $E_i$ and $E_j$ slightly, to obtain transverse 2-cycles $E_i'$ and $E_j'$ depicted as thick arcs in Figure~\ref{figure:W-tilde}, for which we have
  \begin{equation}
    E_i\cdot E_j = E_i'\cdot E_j' = z_i\circ (a_j^{\vphantom1}v_j') =
    a_ia_j\operatorname{lk}_{\Sigma_m}(v_i^{\vphantom1},v_j'),
    \label{equation:intersection-to-linking}   
  \end{equation}
  where $v_j'$ denotes a pushoff of $v_j$ in $\Sigma_m$ taken along the $(-1)$-framing.
  More details are as follows.
  Deform $E_i$ to get a homologous cycle $E_i' = a_i^{\vphantom1}\sigma_i' - z_i'$, where 
  $z_i'$ is obtained by pushing $z_i\subset \Sigma_m=\Sigma_m\times 0 \subset \partial W$ slightly into the interior of $W$, and $\sigma_i'$ is obtained by removing a collar of $\partial \sigma_i=v_i\times 0$ from~$\sigma_i$.
  Also deform $E_j$ to get a homologous cycle $E_j'=a_j^{\vphantom1}\sigma_j' - z_j'$, where $\sigma_j'$ is obtained by pushing the 2-disk $\sigma_j$ along its unique framing in $W$, and $z_j'$ is obtained from $z_j$ by attaching annuli cobounded by $-\partial z_j$ and~$a_j(\partial\sigma_j')$.
  See Figure~\ref{figure:W-tilde}.
  Note that $\partial \sigma_j'=v_j'\times 0$, since the 2-handles are attached along the $(-1)$-framing. 
  Now, $\sigma_i' \cdot \sigma_j' = \sigma_i' \cdot z_j' = z_i' \cdot z_j' =0$, and so $E_i'\cdot E_j' = (-z_i') \cdot (a_j^{\vphantom{1}}\sigma_j')$ in~$W$.
  This intersection is equal to the intersection $-(-z_i) \circ (a_j^{\vphantom{1}} v_j')$ in~$\Sigma_m=\Sigma_m\times 0$ (see Figure~\ref{figure:W-tilde} again), where the additional~$-$ sign is needed since the orientation of $\Sigma_m$ is opposite to the boundary orientation on~$\partial W$.
  From this and the definition of $\lk_{\Sigma_m}$, \eqref{equation:intersection-to-linking} is obtained immediately.

  Using \eqref{equation:definition-of-E0} and
  \eqref{equation:intersection-to-linking}, it follows that
  \eqref{equation:spinc-a} is equivalent to the following linking
  number conditions:
  \begin{gather}
    \label{equation:linking-condition-1}
    \sum_{i=3m-5}^{4m-6} \sum_{j=3m-5}^{4m-6}
    \lk_{\Sigma_m}(v_i^{\mathstrut},v_j') = -m,
    \\
    \label{equation:linking-condition-2}
   \sum_{j=3m-5}^{4m-6}a_ia_j
    \lk_{\Sigma_m}(v_i^{\mathstrut},v_j') \equiv a_i^2 
    \lk_{\Sigma_m}(v_i^{\mathstrut},v_i') \pmod{2} \quad\text{ for all
      $i$}.
  \end{gather}
  
  This reduces the proof to a purely 3-dimensional computation.  It is
  known that the linking number in a rational homology 3-sphere can be
  explicitly computed by using the linking matrix.  (See, for
  instance,~\cite[Theorem~3.1]{Cha-Ko:2000-1}.)  To apply this to our
  case, let $L$ be the link consisting of the 0-framed curves in
  Figure~\ref{figure:covering-surgery-diagram}, from which $\Sigma_m$
  is obtained by surgery.  Orient components of $L$ along the arrows
  in Figure~\ref{figure:covering-surgery-diagram}, and let $x_i$ be
  the positively oriented meridian of the $i$th component (the one
  with label~$x_i$ in Figure~\ref{figure:covering-surgery-diagram}).
  
  Let $P$ be the linking matrix for~$L$.  That is, the $(i,j)$-entry
  is the linking number, in~$S^3$, of the $i$th component and a
  pushoff of the $j$th component taken along the framing.  Then $P$ is
  invertible over $\Q$ since $\Sigma_m$ is a rational homology sphere.  
  With respect to the basis~$\{x_i\}$, $P^{-1}$ gives rise to a well-defined 
  symmetric $\Q$-valued $\Z$-bilinear pairing
  \[R\colon H_1(S^3\sm L)\times H_1(S^3\sm L) \to \Q.\]  That is, define
  $R(x_i,x_j)$ to be the $(i,j)$-entry of $P^{-1}$, and expand it
  bilinearly. Then, for disjoint 1-cycles $\alpha,\beta$ in
  $S^3 \sm L$, we have
  \begin{equation}
    \label{equation:linking-computation}  
    \lk_{\Sigma_m}(\alpha,\beta)=\lk_{S^3}(\alpha,\beta) -
    R(\alpha,\beta).
  \end{equation}

  From Figure~\ref{figure:covering-surgery-diagram}, it is seen that
  $P$ is given as the following matrix, which consists of
  $(m-1)\times (m-1)$ blocks of size $2\times 2$:
  \[
    P={\arraycolsep=4pt\def\arraystretch{1.2}\begin{bmatrix}
        3A_0 & -A_1 & \\
        -A_1^T & 3A_0 & -A_1 \\
        & -A_1^T & 3A_0 & \smash{\ddots} \\
        & & \ddots & \ddots & -A_1 \\
        & & & -A_1 & 3A_0 \\
      \end{bmatrix}},
  \]
  where $A_r:=\sbmatrix{ 0&2^r \\ 1&0}$.  The inverse of $P$ is the
  following symmetric matrix consisting of $(m-1)\times (m-1)$ blocks:
  \begin{equation}
    \label{equation:inverse-of-linking-matrix}
    P^{-1}={
      \arraycolsep=4pt\def\arraystretch{1.2}\frac{1}{2^m-1}\begin{bmatrix}
        c_1c_{m-1}A_0 & c_1c_{m-2}A_1^T & c_1c_{m-3}A_2^T & \cdots & c_1c_1A_{m-2}^T\\
        c_1c_{m-2}A_1 & c_2c_{m-2}A_0 & c_2c_{m-3}A_1^T & \cdots & c_2c_1A_{m-3}^T\\
        c_1c_{m-3}A_2 & c_2c_{m-3}A_1 & c_3c_{m-3}A_0 & \cdots & c_3c_1A_{m-4}^T\\
        \vdots & \vdots &\vdots & \ddots & \vdots \\
        c_1c_1A_{m-2} & c_2c_1A_{m-3} & c_3c_1A_{m-4} & \cdots & c_{m-1}c_1A_0 \\
      \end{bmatrix}
    },
  \end{equation}
  where $c_r:=2^r-1$.  That is, for $k\ge \ell$, the $(k,\ell)$-block
  of $P^{-1}$ is $\frac{c_\ell c_{m-k}}{2^m-1} A_{k-\ell}$.
  A straightforward matrix multiplication confirms $P^{-1} P=I$.  We
  omit the details since it is routine.  (We remark that the
  identities $A_r^2=2^r\cdot I$,
  $A_{k-\ell+1}A_1 = 2A_{k-\ell}A_0=2A_{k-\ell-1}^{\mathstrut}A_1^T$
  and $A_{\ell-k+1}^TA_1^T=2A_{\ell-k}^TA_0=2A_{\ell-k-1}^TA_1$ are
  useful in the verification of $P^{-1}P=I$, for the diagonal,
  below-diagonal and above-diagonal entries respectively.)

  Now, we are ready to verify \eqref{equation:linking-condition-1} and
  \eqref{equation:linking-condition-2}, using
  \eqref{equation:linking-computation} and
  \eqref{equation:inverse-of-linking-matrix}.  First we
  prove~\eqref{equation:linking-condition-1}.  Since each $v_i$ is
  $(-1)$-framed and since $v_i$ and $v_j$ are split for $i\ne j$, we
  have
  \begin{equation}
    \label{equation:linking-of-v_i-v_j-in-S3}
    \lk_{S^3}(v_i^{\mathstrut},v_j')=-\delta_{ij}
    \quad\text{for all $i$ and~$j$}.
  \end{equation}
  If $i\geq 3m-4$, $R(v_i,v_j)=0$ for all $j$, since $v_i$ is
  null-homologous in $S^3 \sm L$.  So, by
  \eqref{equation:linking-computation} and 
  \eqref{equation:linking-of-v_i-v_j-in-S3}, we have
  \begin{equation}
    \label{equation:linking-of-higher-v_i}
    \lk_{\Sigma_m}(v_i^{\mathstrut},v_j') = -\delta_{ij}
    \quad\text{for }i\ge 3m-4.    
  \end{equation}
  From Figure~\ref{figure:covering-surgery-diagram}, $v_{3m-5}=x_1$ in
  $H_1(S^3\sm L)$.  By the definition, $R(x_1,x_1)$ is equal to the
  $(1,1)$-entry of $P^{-1}$ in
  \eqref{equation:inverse-of-linking-matrix}, which is zero.  Thus, by
  \eqref{equation:linking-computation} and
  \eqref{equation:linking-of-v_i-v_j-in-S3}, we have
  \begin{equation}
    \label{equation:self-linking-of-v_3m-5}
    \lk_{\Sigma_m}(v_{3m-5}^{\mathstrut},v'_{3m-5})=-1. 
  \end{equation}
  From this, \eqref{equation:linking-of-higher-v_i} and the fact that
  $\lk_{\Sigma_m}$ is symmetric, it follows that
  \eqref{equation:linking-condition-1} holds.

  It remains to verify \eqref{equation:linking-condition-2}.
  First, for $i\ge 3m-4$, the left and right hand sides of
  \eqref{equation:linking-condition-2} are equal to $1-2^m$ and $-1$,
  respectively, by~\eqref{equation:linking-of-higher-v_i}.  It follows
  that \eqref{equation:linking-condition-2} holds in this case.  Also,
  for $i=3m-5$, both sides of \eqref{equation:linking-condition-2} are
  equal to $-(2^m-1)^2$, by~\eqref{equation:linking-of-higher-v_i}
  and~\eqref{equation:self-linking-of-v_3m-5}.  So
  \eqref{equation:linking-condition-2} holds in this case.

  Suppose $i\le 3m-6$.  In this case, $a_i=2^m-1$.  From
  Figure~\ref{figure:covering-surgery-diagram}, it is seen that $v_i$,
  oriented counterclockwise, is always of the form $v_i=x_o+x_e$, in
  $H_1(S^3\sm L)$, with $o$ odd and $e$ even.  (In fact, $(o,e)$ is of
  the form $(2k-1,2k+2)$ or $(2k+1,2k)$, depending on the choice of
  $i$, but we do not need this information.)
  Using~\eqref{equation:linking-computation},
  \eqref{equation:linking-of-v_i-v_j-in-S3} and $v_{3m-5}=x_1$, we
  have
  \begin{equation}
    \label{equation:linking-of-v_i-v_3m-5}
    (2^m-1)\lk_{\Sigma_m}(v_i,v_{3m-5})=-(2^m-1)R(x_o,x_1)-(2^m-1)R(x_e,x_1).
  \end{equation}
  Inspecting \eqref{equation:inverse-of-linking-matrix} together with
  the matrix $A_r$, entries in the first column of the integer matrix
  $(2^m-1)P^{-1}$ has alternating parity, starting from zero, since
  $c_r$ is always odd.  That is, $(2^m-1)R(x_o,x_1)$ is even and
  $(2^m-1)R(x_e,x_1)$ is odd.  Therefore, from
  \eqref{equation:linking-of-v_i-v_3m-5}, it follows that the left
  hand side of \eqref{equation:linking-condition-2} is odd.

  On the other hand, since the matrix $A_r$ has zero diagonals,
  $P^{-1}$ in~\eqref{equation:inverse-of-linking-matrix} does too, and
  thus $R(x_o,x_o)=R(x_e,x_e)=0$.  So,
  using~\eqref{equation:linking-computation}
  and~\eqref{equation:linking-of-v_i-v_j-in-S3}, we have
  \[
    (2^m-1)\lk_{\Sigma_m}(v_i,v_i') =
    -(2^m-1)-(2^m-1)R(x_o+x_e,x_o+x_e) \equiv 1 \pmod 2.
  \]
  It follows that the right hand side of
  \eqref{equation:linking-condition-2} is odd.  Hence,
  \eqref{equation:linking-condition-2} holds.  This completes the
  proof of Lemma~\ref{lemma:dual-of-w-is-characteristic}.
\end{proof}

\bibliographystyle{amsalpha}
\def\MR#1{}
\bibliography{research}

\end{document}